\documentclass
[
11pt]
{amsart}

\usepackage{hyperref}
\usepackage{cite}
\usepackage{amssymb}
\usepackage{amsmath}
\usepackage{amsthm}
\usepackage{amsfonts}
\usepackage{bbm}
\usepackage{enumitem} 
\usepackage{graphicx}
\usepackage{xcolor}
\usepackage[toc,page]{appendix}
\usepackage{subfigure}
\usepackage{mathtools}
\usepackage[normalem]{ulem}
\usepackage{comment}

\makeatother
\numberwithin{equation}{section}
\newtheorem{theorem}{Theorem}[section]
\newtheorem{assumption}{Assumption}[section]
\newtheorem{proposition}[theorem]{Proposition}
\newtheorem{lemma}[theorem]{Lemma}
\newtheorem{corollary}[theorem]{Corollary}
\theoremstyle{definition}
\newtheorem{definition}[theorem]{Definition}
\newtheorem{example}[theorem]{Example}
\theoremstyle{remark}
\newtheorem{remark}[theorem]{Remark}

\def\R{{\mathbb R}}

\newcommand{\cC}{{\mathcal C}}
\newcommand{\GG}{\mathbb{G}}

\newcommand{\sign}{\operatorname{sign}}
\newcommand{\trace}{\operatorname{tr}}

\newcommand{\bb}{\mathbb}
\newcommand{\cc}{\mathcal}

\newcommand{\Lipb}{{\operatorname{Lip}_b}}
\newcommand{\Span}{\operatorname{Span}}
\newcommand{\He}{\operatorname{He}}
\newcommand{\wt}{\widetilde}

\newcommand{\bx}{\mathbf{x}}
\newcommand{\by}{\mathbf{y}}
\newcommand{\bz}{\mathbf{z}}
\newcommand{\bG}{{\operatorname{G}}}
\newcommand{\bsu}{\operatorname{SU}}
\newcommand{\bsl}{\operatorname{SL}}

\renewcommand{\le}{\leqslant}
\renewcommand{\ge}{\geqslant}

\makeatletter
\@namedef{subjclassname@2020}{%
	\textup{2020} Mathematics Subject Classification}
\makeatother

\begin{document}

\title[Dimension-independent functional inequalities]{Dimension-independent functional inequalities by tensorization and projection arguments}

\author[Baudoin]{Fabrice Baudoin{$^{\ddag}$}}
\address{ Department of Mathematics\\
Aarhus University, Denmark}
\thanks{\footnotemark {$^{\ddag}$} Research was supported in part by NSF Grant DMS-2247117 when most of the work was completed.}
\email{fbaudoin@math.au.dk}

\author[Gordina]{Maria Gordina{$^{\dag}$}}
\address{ Department of Mathematics\\
University of Connecticut\\
Storrs, CT 06269,  U.S.A.}
\thanks{\footnotemark {$\dag$} Research was supported in part by NSF Grants DMS-2246549.}
\email{maria.gordina@uconn.edu}

\author[Sarkar]{Rohan Sarkar{$^{\dag }$}}
\address{ Department of Mathematics\\
University of Connecticut\\
Storrs, CT 06269,  U.S.A.}
\email{rohan.sarkar@uconn.edu}



\date{\today \ \emph{File:\jobname{.tex}}}

\begin{abstract}
We study stability under tensorization and projection-type operations of gradient estimates and other functional inequalities for Markov semigroups on metric spaces. Using transportation inequalities obtained by F.~Baudoin and N.~Eldredge in 2021, we prove that constants in the gradient estimates can be chosen to be dimension-independent. Our results are applicable to hypoelliptic diffusions on sub-Riemannian manifolds and some hypocoercive diffusions. As a byproduct, we obtain dimension-independent reverse Poincar\'{e} inequality, reverse logarithmic Sobolev inequality, and gradient bounds on Lie groups with transverse symmetries and for non-isotropic Heisenberg groups.
\end{abstract}

\maketitle
\tableofcontents

\section{Introduction}

This paper studies stability of functional inequalities satisfied by Markov semigroups on products of metric measure spaces. We focus on three key functional inequalities: gradient bounds, reverse Poincar\'{e} and reverse logarithmic Sobolev  inequalities. These inequalities play a crucial role in establishing various analytical properties for the underlying Markov semigroup, including Liouville-type properties, Harnack-type inequalities, convergence towards the equilibrium distribution (if it exists), quasi-invariance, and more. In this context, the past three decades have witnessed significant progress in proving these functional inequalities on general metric measure spaces equipped with a Dirichlet form, see \cite{VaropoulosBook1992, CoulhonJiangKoskelaSikora2020} and references therein. These results are often interpreted within the framework of Gamma calculus introduced by Bakry and \'{E}mery \cite{BakryEmery1985} and curvature-dimension inequalities. Such techniques are widely applicable in the setting of Riemannian manifolds and elliptic diffusions.
	
However, for hypoelliptic diffusions several fundamental issues arise due to lack of such geometric methods in general, and in particular, for not having a Dirichlet form corresponding to the hypoelliptic differential generator. Such difficulties have been tackled by extending the work of Bakry-\'{E}mery to hypoelliptic settings, see \cite{Baudoin2017a} in the context of obtaining Villani's \cite{Villani2009} hypocoercivity estimates, \cite{BaudoinGordinaMelcher2013, BaudoinGordinaMariano2020} for gradient estimates of Kolmogorov diffusions, \cite{BaudoinGordinaHerzog2021} for Langevin dynamics with singular potentials, \cite{BaudoinBonnefont2009} for gradient estimates of sub-elliptic heat kernels on $\operatorname{SU}(2)$, and using coupling for gradient estimates on the Heisenberg group \cite{BanerjeeGordinaMariano2018}. One of the key tools in such a context is the generalized curvature-dimension condition, which implies reverse Poincar\'e and reverse logarithmic Sobolev inequalities, Li-Yau type gradient estimates for the associated Markov semigroups. For a detailed account on such techniques, we refer to \cite{BaudoinBonnefont2012, BaudoinGarofalo2017}. In general, the results obtained through these methods may lead to dimension-dependent functional inequalities. Besides geometric and probabilistic techniques such as coupling, functional inequalities in such degenerate settings can be approached by using the structure of the underlying spaces as in \cite{CamrudGordinaHerzogStoltz2022, DriverMelcher2005, Gordina2017, EldredgeGordinaSaloff-Coste2018, LiHong-Quan2006, LiHong-QuanZhang2019, GordinaLuo2022, GordinaLuo2024}. 

The present work is closely related to \cite{GordinaLuo2022, GordinaLuo2024}, where the authors employ tensorization and projection-type arguments to derive a dimension-independent logarithmic Sobolev  inequality on non--isotropic Heisenberg groups. They further extend these ideas to homogeneous spaces, utilizing the tensorization property of the Dirichlet form associated with the heat semigroup on such manifolds to establish their results. As mentioned earlier, gradient bounds, reverse Poincar\'{e} and reverse logarithmic Sobolev inequalities cannot be deduced in such a way when we do not have a natural Dirichlet form associated to the hypoelliptic operator. To address this issue, we adopt a different approach by invoking the duality results for these functional inequalities developed in \cite{Kuwada2010a, Kuwada2013b} and \cite{BaudoinEldredge2021}. While the first two papers focus on the relationship between gradient bounds and Wasserstein metric, the latter provides equivalent formulations of reverse Poincar\'e and reverse logarithmic Sobolev inequalities through the lens of Wasserstein and Hellinger metrics, along with some entropy inequalities. This approach enables us to extend our results beyond manifolds, specifically to any path-connected metric measure space, also known as \emph{length space}. Using simple inequalities involving the Wasserstein and Hellinger metrics, we show that \eqref{eq:GB}, \eqref{eq:RP}, and \eqref{eq:RLS} in Theorem~\ref{thm:main} extend to the product space, and the constants in these inequalities can be chosen not to depend on the number of spaces in the product. As a byproduct of our results, we show that functional inequalities including Li-Yau type gradient estimates and parabolic Harnack inequalities can be easily deduced for sub-Riemannian manifolds obtained through tensorization and \emph{sub-Riemannian submersion} as defined in  Definition~\ref{def:submersion}. In Section~\ref{sec:3} we consider $(2n+m)$-dimensional Lie groups with transverse symmetry, which is a sub-Riemannian manifold with transverse symmetry introduced in \cite{BaudoinGarofalo2017}. In Theorem~\ref{thm:transverse} we show that in the context of Lie groups with transverse symmetry, our approach yields functional inequalities sharper than those obtained using the curvature-dimension criterion in \cite{BaudoinGarofalo2017}. Our examples include Kolmogorov diffusions on $\mathbb{R}^d\times \mathbb{R}^d$, kinetic Fokker-Planck operators, Carnot groups of step $2$ including non-isotropic Heisenberg groups, the orthogonal groups $\operatorname{SO}(3), \operatorname{SO}(4)$, and two sub-Riemannian manifolds which are not Lie groups: the Heisenberg nil-manifold and the Grushin plane.
	
The  paper is organized as follows. In Section~\ref{sec:1} we describe the framework for our main results. Section~\ref{sec:2} presents preliminaries on sub-Riemannian manifolds, along with the implications of Theorem~\ref{thm:main} within the context of sub-Riemannian geometry. Section~\ref{sec:4} presents examples including Section~\ref{sec:3} devoted to Lie groups with transverse symmetry.
	
\section{Tensorization of functional inequalities on metric measure spaces}\label{sec:1}
	
Let $\left(X, d \right)$ be a complete, locally compact, separable metric space which is a length space. In particular, by the Hopf-Rinow theorem for length spaces (see e.g. \cite[p.~9]{GromovBook2007}), each pair of points can be joined by a minimizing geodesic, i.~e. a rectifiable curve whose length is the distance between the points.

We assume that $X$ is equipped with a \emph{strong upper gradient}, which implies that for any measurable function $f:X \to \bb{R}$ we have the upper Lipschitz constant of $f$ defined by
\begin{align}
|\nabla f|(x)=\lim_{r\downarrow 0} \sup_{y: d(x,y)<r}\frac{|f(x)-f(y)|}{d(x,y)}.
\end{align}
We refer to \cite{Cheeger1999} for some basic properties of length spaces and strong upper gradients, and for more details \cite[p.~2]{AmbrosioGigliSavareBook2005} and \cite[Section 6.2]{HeinonenKoskelaShanmugalingamTysonBook2015}. We denote by $\Lipb(X)$ the Banach space of all Lipschitz functions on $X$ endowed with the Lipschitz norm 

\begin{align*}
\|f\|_{\Lipb(X)}=\sup_{x\in X}|f(x)|+\sup_{x\neq y}\left|\frac{f(x)-f(y)}{d(x,y)}\right|,
\end{align*}
by $\cc{B}_X$ we denote the Borel $\sigma$-algebra on $(X, d)$, and by $\cc{P}(X)$ the space of all probability measures on $\left( X, \cc{B}_X \right)$. For a Markov kernel $P: X\times\cc{B}_X\to [0,1]$, we denote by $Pf$ and $\mu P$ the usual action of $P$ on bounded Borel functions and probability measures, that is, 
\begin{align*}
& Pf(x):=\int_X f(y) P(x, dy), 
\\
& \mu P(A):=\int_X P(x, A)\mu(dx).
\end{align*}
We are interested in the following functional inequalities for $P$.
\begin{enumerate}[leftmargin=*]
\item \emph{Gradient bound:} for $1\leqslant p\leqslant 2$ and for all $f\in\Lipb(X)$ 
\begin{align}\label{eq:GB}
		|\nabla P f| \leqslant C (P|\nabla f|^p)^{\frac{1}{p}}.\tag{$\operatorname{GB}_p$}
\end{align}
\item \emph{Reverse Poincar\'e inequality:} for all $f\in \Lipb(X)$
	\begin{align} \label{eq:RP}
		|\nabla Pf|^2\leqslant C(P(f^2)-(Pf)^2). \tag{$\operatorname{RPI}$}
	\end{align} 
\item \emph{Reverse logarithmic Sobolev  inequality:} for all positive $f\in \Lipb(X)$
	\begin{align}\label{eq:RLS}
		Pf |\nabla \log Pf|^2 \leqslant C(P(f\log f)- Pf\log(Pf)) \tag{$\operatorname{RLSI}$},
	\end{align}
	\end{enumerate}
where $C$ is a positive constant and may differ for each inequality.

We now consider a collection of complete separable metric length spaces $\left\{ (X_i, d_i) \right\}_{i=1}^n$ and let $X:=X_1\times \cdots\times X_n$ be endowed with the  metric $d$ defined by
\begin{align}\label{eq:square_metric}
	d(x, y)^2:= \sum_{i=1}^n d_i (x_i, y_i)^2.
\end{align}
It is known that $(X,d)$ is a complete separable metric length space. For Markov kernels $P_i$ defined on $(X_i, d_i)$, the tensor product $P$ is defined as the Markov operator on $X=X_1\times\cdots\times X_n$ such that for all $f \in B_b(X)$
\begin{align*}
& Pf:=P_1\otimes\cdots\otimes P_n f(x_1,\ldots, x_n)
\\
& =\int_{X} f(z_1,\ldots, z_n) P_1(x_1, dz_1)\cdots P_n(x_n, dz_n).
\end{align*}
We are now ready to state the main result of this section.
\begin{theorem}\label{thm:main}
Let $P_i$ be a Markov kernel on $(X_i, d_i), i=1, ..., n$ that satisfies \eqref{eq:GB} (resp. \eqref{eq:RP}, \eqref{eq:RLS}) with a constant $C_i$. Then the Markov kernel $P:=\otimes_{i=1}^n P_i$ satisfies \eqref{eq:GB} (resp. \eqref{eq:RP}, \eqref{eq:RLS}) with constant $C=\max\left\{C_i: 1 \leqslant i\leqslant n \right\}$.
\end{theorem}
By \cite[Theorem~3.5]{BaudoinEldredge2021} we know that \eqref{eq:RP} is equivalent to a Harnack-type inequality, thus we obtain the following corollary.
\begin{corollary}
Let $P_i$ be a Markov kernel on $(X_i, d_i), i=1, ..., n$ such that for all bounded non-negative $f \in B_b(X_i)$ and $x_i, y_i\in X_i$,
	\begin{align*}
		Pf(x_i)\leqslant Pf(y_i)+\sqrt{C_i} d_i(x_i, y_i)\sqrt{P(f^2)(x_i)}
	\end{align*}
for some constant $C_i>0$. Then for all $f\in B_b(X)$ and $x, y \in X$ we have
\begin{align}
	Pf(x)\leqslant Pf(y)+\sqrt{C} d(x, y)\sqrt{P(f^2)(x)}
\end{align}
with $C=\max\{C_i: 1\leqslant i\leqslant n\}$.
\end{corollary}

\subsection{Transportation inequalities and tensorization}
Our goal in this section is to prove Theorem~\ref{thm:main} using transportation inequalities. 
Suppose as before that $\left(X, d \right)$ is  a complete, locally compact, separable length space. 
For any $1 \leqslant p\le \infty$, the \emph{$L^{p}$-Wasserstein distance} $W_p(\mu, \nu)$ between $\mu, \nu\in\cc P(X)$ is defined as 
\begin{align}\label{eq:W_p_def}
W_p(\mu, \nu):=\inf_{\pi\in \mathcal{C}\left( \mu, \nu \right)} \|d\|_{L^p(X\times X, \pi)},
\end{align}
where $\mathcal{C}\left( \mu, \nu \right)$ is the collection of all possible couplings of $\mu,\nu\in\cc P(X)$. $W_p$ defines a metric on the $\cc P(X)$, although $W_p$ can be infinite. When $W_p(\mu,\nu)<\infty$, the infimum in \eqref{eq:W_p_def} is attained, that is, there exist $X$-valued random variables $U, V$ such that $U\sim \mu, V\sim \nu$ and $\mathbb{E}(d(U,V)^p)=W_{p}(\mu, \nu)^p$. We refer to \cite[Chapter 6]{VillaniOptimalTransportBook} for a discussion about Wasserstein distances and Wasserstein spaces.

Next we introduce the \emph{Hellinger distance} between $\mu,\nu\in \cc{P}(X)$ defined by
\begin{align*}
\He_2(\mu, \nu)^2:=\int_X\left(\sqrt{\frac{d\mu}{dm}}-\sqrt{\frac{d\nu}{dm}}\right)^2 dm,
\end{align*}
for some measure $m$ with respect to which both $\mu,\nu$ are absolutely continuous. This definition is independent of $m$ and in particular one can take $m=(\mu+\nu)/2$. Before going into the proof of Theorem~\ref{thm:main}, we prove the following lemma which will be crucial subsequently.

\begin{lemma}\label{lem:wass_hell}
For each $i=1,\ldots, n$, let $\mu_i, \nu_i \in \cc{P}(X_i)$. Then 

\begin{align}
&
W_p(\mu_1\otimes\cdots\otimes \mu_n, \nu_1\otimes \cdots\otimes\nu_n)^2\leqslant \sum_{i=1}^n W_p(\mu_i,\nu_i)^2,  \ p \geqslant 2 ; \label{it:wass}
\\
&
\He_2(\mu_1\otimes\cdots\otimes\mu_n, \nu_1\otimes\cdots\otimes\nu_n)^2\leqslant \sum_{i=1}^n\He_2(\mu_i, \nu_i)^2, \label{it:hell}	
\end{align}
where $\mu_1\otimes\cdots\otimes\mu_n, \nu_1\otimes\cdots\otimes\nu_n$ are product measures on $(X,d)$.
\end{lemma}
\begin{proof}
For each $i=1,\ldots, n$, let $(U_i, V_i)$ be an optimal coupling such that $U_i\sim \mu_i, V_i\sim \nu_i$ and $\|d_i(U_i, V_i)\|_p=W_p(\mu_i,\nu_i)$, where for a random variable $Z$, $\|Z\|_p=(\mathbb{E}|Z|^p)^{1/p}$ denotes the $L^p$ norm of $Z$. Without loss of generality, we can assume that $(U_i,V_i)_{i=1}^n$ are jointly independent pairs. Define $U=(U_1,\ldots, U_n), V=(V_1,\ldots, V_n)$. Now using the triangle inequality in the $L^{p/2}$ space, we have

\begin{align*}
& \|d(U,V)\|^2_p=\left\|\sum_{i=1}^n d_i(U_i, V_i)^2\right\|_{p/2}
\\
& \leqslant \sum_{i=1}^n \|d_i(U_i, V_i)^2\|_{p/2}=\sum_{i=1}^n W_p(\mu_i, \nu_i)^2.
\end{align*}
Since $W_p(\mu_1\otimes\cdots\otimes\mu_n, \nu_1\otimes\cdots\otimes \nu_n)\leqslant \|d(U, V)\|_{p}$, then \eqref{it:wass} follows. 
	
For \eqref{it:hell}, let $m_1, \ldots, m_n$ be probability measures on $(X,d)$ such that both $\mu_i, \nu_i$ are absolutely continuous with respect to $m_i$ for all $1\leqslant i\leqslant n$. Writing $f_i:=\frac{d\mu_i}{dm_i}, g_i:=\frac{d\nu_i}{dm_i}$, we have

\begin{align}
\He_2(\mu_1\otimes\cdots\otimes\mu_n, \nu_1\otimes\cdots\otimes\nu_n)^2&=\int\left(\sqrt{f_1\cdots f_n}-\sqrt{g_1\cdots g_n}\right)^2dm_1\cdots dm_n \notag 
\\
&=2-2\int\sqrt{f_1\cdots f_n g_1\cdots g_n}dm_1\cdots dm_n  \notag
\\
&=2-2\prod_{i=1}^n \int \sqrt{f_ig_i}dm_i \label{eq:he1}
\\
&\leqslant \sum_{i=1}^n \ 2\left(1-\int\sqrt{f_ig_i}dm_i\right) \notag
\\
&=\sum_{i=1}^n \He_2(\mu_i, \nu_i)^2, \label{eq:he2}
\end{align}
where \eqref{eq:he2} follows from \eqref{eq:he1} using  the elementary fact that
\[
1-x_1\cdots x_n\leqslant \sum_{i=1}^n (1-x_i) \text{ for } 0\leqslant x_1,\ldots, x_n\leqslant 1.
\]
\end{proof}

\begin{proof}[Proof of Theorem~\ref{thm:main}]
To prove stability of \eqref{eq:GB} under tensorization, we use Kuwada's duality theorem \cite[Thereom~2.2]{Kuwada2010a}. In that paper, the duality result was obtained under the assumption of a volume doubling property, which was later relaxed and generalized to Orlicz spaces in \cite{Kuwada2013b}. Moreover, according to the proof of \cite[Theorem~2.2]{Kuwada2010a}, to deduce \eqref{eq:GB} it suffices to show that for any $x, y\in X$ 
\begin{align}\label{eq:GB_P}
W_q(\delta_x P, \delta_y P)\leqslant C d(x,y),
\end{align}
where $\delta_a$ denotes the Dirac measure at $a\in X$ and $p^{-1}+q^{-1}=1$. Since we are considering $1\leqslant p\leqslant 2$, it follows that $q\geqslant 2$. Now for each $1\leqslant i\leqslant n$, and $x_i, y_i\in X_i$, \cite[Theorem~2.2]{Kuwada2010a} implies that $W_q(\delta_{x_i} P_i, \delta_{y_i}P_i)\leqslant C_i d_i(x_i, y_i)$. Using \eqref{it:wass} in Lemma~\ref{lem:wass_hell} we conclude
\begin{align*}
W_q(\delta_x P, \delta_y P)^2&=W_q(\otimes_{i=1}^n \delta_{x_i}P_i, \otimes_{i=1}^n \delta_{y_i} P_i)^2 \\
&\leqslant \sum_{i=1}^n W_q(\delta_{x_i} P_i, \delta_{y_i} P_i)^2 
\\
&\leqslant \sum_{i=1}^n C^2_i d_i(x_i, y_i)^2  \leqslant C^2 d(x,y)^2,
\end{align*}
where $C=\max\{C_i: 1\leqslant i\leqslant n\}$.

Next, to prove stability of \eqref{eq:RP} under tensorization, we use the Hellinger-Kantorovich contraction criterion from \cite[Theorem~3.7]{BaudoinEldredge2021}. Following the argument in their proof, it is enough to show that for any $x, y\in X$
\begin{align}\label{eq:RPI_P}
\He_2(\delta_x P, \delta_y P)^2\leqslant \frac{C}{4} d(x,y)^2.
\end{align}
Since \eqref{eq:RP} holds for each $P_i$ with constant $C_i$, \cite[Theorem~3.7]{BaudoinEldredge2021} implies that for all $x_i, y_i\in X_i$ 
\begin{align*}
\He_2(\delta_{x_i} P_i, \delta_{y_i} P_i)^2 \leqslant \frac{C_i}{4} d_i(x_i, y_i)^2.
\end{align*}
Therefore, using \eqref{it:hell} in Lemma~\ref{lem:wass_hell} and an argument similar to the proof of \eqref{eq:GB_P}, we conclude that \eqref{eq:RPI_P} holds with $C=\max\{C_i: 1\leqslant i\leqslant n\}$, which is equivalent to \eqref{eq:RP} for $P$. 

Finally, to prove \eqref{eq:RLS} for $P$, we again resort to \cite[Theorem~5.15]{BaudoinEldredge2021} which proves equivalence between a reverse logarithmic Sobolev  inequality and a Wang-Harnack inequality. Using the above result, we have that for all $1\leqslant i\leqslant n$ and $p>1$
\begin{gather}\label{eq:WHI_n}
P_i f(x_i)^p \leqslant P_i f^p(y_i)\exp\left(\frac{p}{p-1} \frac{C_id_i(x_i, y_i)^2}{4}\right) 
\\
\text{ for all } x_i, y_i\in X_i,  f\in\Lipb(X_i),  f>0. \notag
\end{gather}
Now, for any positive function $f\in\Lipb(X)$ and $x_i\in X_i$ for $1\le i\le n-1$, we claim that 
\begin{align*}
z\longmapsto &P_1 \otimes \cdots \otimes P_{n-1} f(x_1,\ldots, x_{n-1}, z) 
\\
&:=\int\limits_{X_1\times\cdots\times X_{n-1}} f(y_1,\ldots, y_{n-1},z) P_1(x_1, dy_1)\cdots P_{n-1}(x_{n-1}, dy_{n-1})
\end{align*}
is a positive $\Lipb(X_n)$ function. It is enough to verify our claim for $n=2$, as the the general case will follow by induction. Indeed, for any $f\in\operatorname{Lip}_b(X)$ with $|f(x)-f(y)|\leqslant C_f d(x,y)$ for some constant $C_f>0$ and for all $x,y\in X$, one has
\begin{align*}
&|P_1 f(x_1,x_2)-P_1 f(x_1, y_2)|
\\
& \leqslant \int_{X_1} |f(z, x_2)-f(z, y_2)| P_1(x_1, dz)\leqslant C_f d_2(x_2, y_2),
\end{align*}
where we used the fact that $d((z,x_2), (z,y_2))=d_2(x_2,y_2)$. This proves our claim.
As a result, applying \eqref{eq:WHI_n} we obtain
\begin{align*}
& Pf(x_1,\ldots, x_n)^p =\left(\int_{X_n}P_1\otimes \cdots\otimes P_{n-1} f(x_1,\ldots, x_{n-1}, z)P_n(x_n, dz)\right)^p 
\\
&\leqslant \int_{X_n}\left(P_1\otimes \cdots\otimes P_{n-1} f(x_1,\ldots, x_{n-1}, z)\right)^p P_n(y_n, dz)
\\
& \times \exp\left(\frac{p}{p-1} \frac{C_nd_n^2(x_n, y_n)}{4}\right).
\end{align*}
Iterating the above inequality for each of the variables $x_1,\ldots, x_{n-1}$, we get 
\begin{align*}
& Pf(x)^p\leqslant P f^p(y)\exp\left( \frac{p}{p-1}\sum_{i=1}^n \frac{C_id^2_i(x_i, y_i)}{4}\right)
\\
& \leqslant P f^p(y)\exp\left( \frac{p}{p-1}\frac{C d^2(x,y)}{4}\right)
\end{align*}
with $C=\max\{C_i: 1\leqslant i\leqslant n\}$. Invoking \cite[Theorem 5.15]{BaudoinEldredge2021} we conclude stability of \eqref{eq:RLS} under tensorization. This completes the proof of the theorem. 
\end{proof}

\section{Applications to sub-Riemannian manifolds}\label{sec:2}

Let $M$ be a connected smooth manifold equipped with a sub-Riemannian structure $(\cc H, g)$, where $\cc H$ is a horizontal distribution which satisfies  H\"ormander's condition and is equipped with an inner product $g(\cdot, \cdot)$. For any two points $x, y\in M$, the Carnot-Carath\'eodory distance between $x,y$ is given by 
\begin{align}\label{eq:CC-metric} 
d(x,y):= & \inf\left\{\int_0^1\|\sigma^{\prime}(t)\|_{\cc H} dt: \sigma(0)=x, \sigma(1)=y, 
\right.
\notag
\\
& \left.\sigma^{\prime}(t)\in \cc{H}(\sigma(t)) \text{ for all } 0\leqslant t\leqslant 1\right\},
\end{align}
where for a horizontal vector $u$, $\| u\|_{\cc H}:=\sqrt{g(u,u)}$ and $\sigma: [0,1]\to M$ is absolutely continuous. Throughout this section, we always make the following assumption when considering the Carnot-Carath\'eodory distance. For more on this property we refer to \cite[Section~3.3.1]{AgrachevBarilariBoscainBook2020}.

\begin{assumption}\label{completeness}
$(M, d)$ is a complete metric space.
\end{assumption}
For any smooth function $f$ on $M$, the \emph{horizontal gradient} $\nabla_{\cc H} f$ is defined as the unique element in $\cc H$ such that 
\begin{align*}
	g(\nabla_{\cc H} f, X)= X(f) \text{ for any }  X\in\cc H,
\end{align*}
and the norm of the horizontal gradient is denoted by  
\[
\|\nabla_{\cc H}f\|_{\cc H}^2=g(\nabla_{\cc H} f, \nabla_{\cc H} f).
\]
Suppose  $\mu_{M}$ is a smooth measure on $M$. Now we can define the \emph{divergence} of a smooth vector field $X$ with respect to the measure $\mu_{M}$ as the smooth function $\operatorname{div}_{\mu_{M}} X$ such that

\[
\mathcal{L}_{X}\mu_{M}=\operatorname{div}_{\mu_{M}}\left( X \right) \mu_{M},
\]
where $\mathcal{L}_{X}$ is the Lie derivative in the direction of $X$. For any $f \in C^{\infty}\left( M \right)$ we define 

\[
\Delta_{\cc H}^M:=\operatorname{div}_{\mu_{M}}\left( \nabla_{\cc H} f \right).
\]
This is a sub-Laplacian  whose definition depends on the sub-Riemannian structure on $M$ and the measure $\mu_{M}$. If $X_{1}, ..., X_{m}$ is an orthonormal frame for the horizontal distribution, we can write the horizontal gradient and the sub-Laplacian as 
\begin{align}
& \nabla_{\cc H} f=\sum_{i=1}^{m} X_{i}\left( f \right) X_{i},
\notag
\\
& \Delta_{\cc H}^M f :=\sum_{i=1}^{m} X_{i}^{2}\left( f \right) + \operatorname{div}_{\mu_{M}}\left( X_{i} \right)X_{i}\left( f \right), f \in C^{\infty}\left( M \right).  
\label{e.SubLaplacianLocal}
\end{align}
Note that  we can use the product rule for vector fields to see that $\Delta_{\cc H}^M$ is compatible with the sub-Riemannian structure in the sense that for every $f \in C^\infty(M)$ we have
\begin{equation}\label{e.sRcompatibility}
\frac{1}{2} \Delta_{\cc H}^M (f^2)-f\Delta_{\cc H}^M f=\|\nabla_{\cc H}f\|_{\cc H}^2.
\end{equation}
Observe that this property does not depend on the choice of measure $\mu_M$, as it only uses the second order terms of $\Delta_{\cc H}^M$.

We denote by $L^{2}\left( M, \mu_M \right)$ the space of real-valued functions on $M$ which are
square-integrable with respect to the measure $\mu_{M}$. Then we see that  $\Delta_{\cc H}^M$ is symmetric with respect to $\mu_M$ on $L^{2}\left( M, \mu_M \right)$, that is,  for every $f, g \in \cC^\infty_c(M)$
\[
\int_M f  \Delta_{\cc H}^M  g d\mu_M=\int_M g  \Delta_{\cc H}^M  f d\mu_M.
\]
Moreover, by \cite{Strichartz1986a} we see that completeness of $\left( M, d \right)$ (Assumption~\ref{completeness}) implies that $\Delta_{\cc H}^M$ is essentially self-adjoint on $L^{2}\left(M, \mu_M \right)$, and by \cite{JerisonSanchez-Calle1986, FeffermanPhong1983} it is locally sub-elliptic in addition to being hypoelliptic by \cite{Hormander1967a}. 
 
We refer to \cite{GordinaLaetsch2016, GordinaLaetsch2017} for a discussion of the role of a measure on sub-Riemannian manifolds,  and to \cite[p. 950]{DriverGrossSaloff-Coste2009a} about essential self-adjointness of such operators in the setting of Lie groups, where $\mu_M$ is chosen to be a Haar measure. For a survey on such operators and smooth symmetrizing measure we refer to \cite[Section 1.3]{Baudoin2022}. 

The fact that $\Delta_{\cc H}^M$ is essentially self-adjoint implies that the unique self-adjoint (Friedrichs) extension of $\Delta_{\cc H}^M$ (which we still denote by $\Delta_{\cc H}^M$) is the generator of the Dirichlet form $\mathcal{E}_M$ which is the closure of the quadratic form $\int_M \| \nabla f \|_{\cc H}^2 d\mu_M$, $f \in \cC^\infty_c(M)$. We call the semigroup  corresponding to $\Delta_{\cc H}^M$ the \emph{horizontal heat semigroup} $P^M_t$ as in \cite[Section 1.5]{Baudoin2022}.

\subsection{Sub-Riemannian manifolds obtained by tensorization and projection}\label{sec:3.1}
We introduce the concept of sub-Riemannian submersion, which is a generalization of Riemannian submersion for Riemannian manifolds. 

\begin{definition}\label{def:submersion}
A mapping $\pi:(M, \cc{H}_M, g_M)\longrightarrow (N, \cc{H}_N, g_N)$ between two sub-Riemannian manifolds is called a \emph{sub-Riemannian submersion} if 
\begin{enumerate}[label=(\roman*)]
\item $\pi$ is a submersion between differentiable manifolds $M$ and $N$;
\item $\pi_{\ast}\left( \cc{H}_{M} \right)=\cc{H}_N$;
\item \label{item3} $\pi^{\ast}g_M=g_N$, that is, for all $x\in M$, $d\pi_x: \cc{H}_M(x)\to \cc{H}_N(\pi(x))$ is an isometry. 
\end{enumerate}
\end{definition}
\begin{remark}\label{remark3.2}
Note that the third condition in this definition can be weakened similarly to the proof of \cite[Proposition 3.6]{GordinaLuo2024}as follows. If $\wt{X}_1, \ldots, \wt{X}_n$ is an orthonormal frame of $\cc{H}_M$, then 
\begin{align*}
& \left\{d\pi_x(\wt{X}_i)(\pi(x)): i=1,\ldots, n \text{ such that } \right.
\\
& \left. d\pi_x(\wt X_i)(\pi(x)) \text{ are linearly independent in } \cc{H}_N\right\}
	\end{align*}
forms an orthonormal frame in $\cc{H}_N$.
\end{remark}

As we usually consider sub-Riemannian manifolds equipped with a symmetrizing measure for the sub-Laplacian, we assume that the sub-Riemannian submersion satisfies the following additional assumption. Note that the sub-Laplacians $\Delta^M_{\cc H}$, $\Delta_{\cc H}^N$ depend on the symmetrizing measures, as one can see easily from the expression for sub-Laplacians in local coordinates \eqref{e.SubLaplacianLocal}. 

\begin{assumption}\label{SubmersionMeasure}
Suppose $\pi: M \longrightarrow N$ is a sub-Riemannian submersion between sub-Riemannian manifolds 
$\pi:(M, \cc{H}_M, g_M)\longrightarrow (N, \cc{H}_N, g_N)$, $\Delta_{\cc H}^M$ and $\Delta_{\cc H}^N$ are sub-Laplacians compatible with the corresponding sub-Riemannian structures on $M$ and $N$ respectively. We assume $\Delta^M_{\cc H}(f\circ\pi)=(\Delta^N_{\cc H} f)\circ\pi$ for every $f \in \cC^\infty_c(N)$.
\end{assumption}

\begin{example}[Riemannian submersions]
If $(M,g_M)$ and $(N,g_N)$ are Riemannian manifolds, then a Riemannian submersion $\pi:(M,g_M) \to (N,g_N)$ induces a sub-Riemannian submersion $\pi:(M, \cc{H}_M, g_M)\longrightarrow (N, \cc{H}_N, g_N)$,  where $\cc{H}_M$ is the horizontal space of the submersion which we assume to satisfy H\"ormander's condition and $\cc{H}_N$ is the whole tangent bundle of $N$. The horizontal Laplacian $\Delta^M_\mathcal{H}$ on $M$, as defined on \cite[p.~70]{BaudoinDemniWangBook2024}, is then a sub-Laplacian on $M$ with symmetrizing measure $\mu_M$, where $\mu_M$ is the Riemannian volume measure of $M$. In this case $\pi$ intertwines the horizontal Laplacian  $\Delta^M_\mathcal{H}$ with the Laplace-Beltrami operator $\Delta^N_\mathcal{H}$  of $N$ as in  Assumption \ref{SubmersionMeasure} if and only if $\pi$ is harmonic, see \cite[Theorem~1]{GoldbergIshihara1978} and the proof of \cite[Theorem 4.1.10.]{BaudoinDemniWangBook2024}. It follows for instance  that for Riemannian submersions with totally geodesic fibers Assumption~\ref{SubmersionMeasure} is satisfied when $\Delta^M_{\mathcal H}$ is the horizontal Laplacian and $\Delta^N_\mathcal{H}$ is the Laplace-Beltrami operator.  
\end{example}

\begin{example}\label{example3.3}[Measure preserving sub-Riemannian submersions]
Suppose $\pi: M \longrightarrow N$ is a sub-Riemannian submersion between sub-Riemannian manifolds 
$(M, \cc{H}_M, g_M)$ and $(N, \cc{H}_N, g_N)$, $\Delta_{\cc H}^M$ and $\Delta_{\cc H}^N$ are sub-Laplacians compatible with corresponding sub-Riemannian structures on $M$ and $N$ respectively. If $\mu_M$ is a symmetrizing measure for $\Delta_{\cc H}^M$, and $\mu_N$ is a symmetrizing measure for $\Delta_{\cc H}^N$, and we assume that the measure $\mu_N$ is the pushforward of $\mu_M$ under $\pi$, i.e. $\mu_N=\pi_{\sharp} \mu_M$, then Assumption~\ref{SubmersionMeasure} holds. Indeed, since $\pi$ satisfies \ref{item3} in Definition~\ref{def:submersion}, an argument similar to \cite[Equation~(4.3)]{GordinaLuo2022} and \cite[Equation~(3.4)]{GordinaLuo2024} leads to 
\begin{align*}
\|\nabla^M_{\cc H}(f\circ \pi)\|^2_{\cc H_M}=\|\nabla^N_{\cc H} f\|_{\cc H_N}^2\circ\pi,
\end{align*}
where $\nabla^M_{\cc H}$ (resp. $\nabla^N_\cc{H}$) denotes the horizontal gradient on $(M,\cc H_M)$ (resp. $N, \cc{H}_N$).
Moreover, since $\mu_N=\pi_{\sharp} \mu_M$ we deduce that  for any $f\in \cC_c^\infty(N)$
\[
\int_M \|\nabla^M_{\cc H}(f\circ \pi)\|^2_{\cc H_M} d\mu_M=\int_N \|\nabla^N_{\cc H} f \|^2_{\cc H_N} d\mu_N.
\]
As a consequence, one has that for any $f\in \cC^\infty_c(N)$, $\Delta^M_{\cc H}(f\circ\pi)=(\Delta^N_{\cc H} f)\circ\pi$.
\end{example}

\begin{example}[Sub-Riemannian isometries]
Suppose $\pi: M \longrightarrow N$ is a sub-Riemannian isometry between sub-Riemannian manifolds 
$(M, \cc{H}_M, g_M)$ and $(N, \cc{H}_N, g_N)$. This means that $\pi$ is a diffeomorphism such that $d\pi$ is an isometry between $(\cc{H}_M, g_M)$ and $(\cc{H}_N, g_N)$. Assume that $\Delta_{\cc H}^M$ and $\Delta_{\cc H}^N$ are  sub-Laplacians  on $M$ and $N$ respectively which admit Popp's measures on $M$ and $N$ respectively as symmetrizing measures, see \cite[Corollary 2]{BarilariRizzi2013}. In that setting, it follows that from \cite[Proposition~7]{BarilariRizzi2013} that sub-Riemannian  isometries are volume preserving transformations for Popp's volume, and therefore Assumption \ref{SubmersionMeasure} holds.
\end{example}

\begin{example}[Sub-Riemannian submersions between Lie groups]
Suppose $\pi: M \longrightarrow N$ is a sub-Riemannian submersion between sub-Riemannian manifolds 
$(M, \cc{H}_M, g_M)$ and $(N, \cc{H}_N, g_N)$, where $M$ and $N$ are Lie groups equipped with left-invariant sub-Riemannian structures. Consider the sub--Laplacians $\Delta_{\cc H}^M$ and $\Delta_{\cc H}^N$ whose symmetrizing measures are Haar measures on $M$ and $N$ respectively. Then Assumption \ref{SubmersionMeasure} holds if $M$ and $N$ are unimodular. Indeed, by \cite[Proposition~17]{Agrachev} and \cite[Theorem~4.3]{GordinaLaetsch2016} for unimodular groups these sub-Laplacians can be written as  
\[
\Delta_{\cc H}^M=\sum_{i=1}^d X_i^2,
\]
where $\left\{ X_i \right\}_{i=1}^{d}$ form a left-invariant orthonormal frame of $\cc{H}_M$. It then follows from the chain rule that 
\begin{align*}
\Delta^M_{\cc H}(f\circ\pi)& =\left(\sum_{i=1}^d X_i^2\right) (f\circ\pi) \\
 & =\left[\left(\sum_{i=1}^d d\pi(X_i)^2\right)f \right] \circ \pi 
\end{align*}
Since the sub-Riemannian submersion property implies that $\left\{d\pi(X_i) \right\}_{i=1}^{d}$ form a left-invariant orthonormal frame of $\cc{H}_M$, we have that $\sum_{i=1}^d d\pi(X_i)^2 =\Delta_{\cc H}^N$ which proves that Assumption~\ref{SubmersionMeasure} holds.
\end{example}

\begin{example}[Homogeneous spaces]\label{example:homogeneous_space}
Let $G$ be a connected Lie group of dimension $(n+m)$ and $H$ be a closed subgroup of $G$. Then by \cite[Theorem~20.12]{LeeBook2003SmoothManifold}, $H$ is an embedded sub-manifold of $G$. If $H$ is a $k$-dimensional sub-manifold of $G$, the right cosets $H\backslash G$ of $H$ have an induced smooth structure and form an $(n+m-k)$ dimensional smooth manifold. We call $H\backslash G$ a \emph{homogeneous space} and denote it by $M$. There exists a natural submersion $\pi:G\longrightarrow M$ defined by $\pi(g)=Hg$. If $(G,\cc{H}_G, \langle\cdot, \cdot\rangle_{\cc{H}_G})$ is a sub-Riemannian structure on $G$, then by \cite[Theorem~3.2]{GordinaLuo2024}, the distribution 
				\begin{align*}
					\cc{H}_M(\pi(g))=\Span\{d\pi_g(\wt{X})(\pi(g)): \wt X\in\cc{H}_G \mbox{ such that $d\pi_g(\wt{X})(\pi(g))\neq 0$}\}
				\end{align*}	
satisfies H\"ormander's condition. Moreover, according to the proof of \cite[Theorem~3.2]{GordinaLuo2024}, if $\{\wt{X}_1,\ldots, \wt{X}_n\}$ is an orthonormal frame of $\cc{H}_G$, then we can define a metric $\langle\cdot, \cdot\rangle_{\cc{H}_M}$ on $\cc{H}_M$ with respect to which $\{d\pi_g(\wt{X}_i)(\pi(g)): i=1,\ldots, n \mbox{ such that $d\pi_g(\wt{X}_i)(\pi(g))$ are linearly independent}\}$ forms an orthonormal frame for $\cc{H}_M$. Hence, $(M,\cc{H}_M, \langle\cdot, \cdot\rangle_M)$ defines a sub-Riemannian structure on $M$, and according to Remark~\ref{remark3.2}, $\pi: (G,\cc{H}_G, \langle\cdot, \cdot\rangle_G)\longrightarrow (M,\cc{H}_M, \langle\cdot, \cdot\rangle_M)$ is a sub-Riemannian submersion. For any $f\in C^\infty(M)$, the horizontal gradient on $M$ is defined as 
\begin{align*}
\nabla^M_{\cc H} f(\pi(g))=d\pi_g(\nabla^G_{\cc H}(f\circ\pi))(\pi(g)) \text{ for all } g\in G.
\end{align*}
One can define the sub-Laplacian on $M$ by 
\begin{align*}
\Delta^M_\cc{H}=\sum_{i=1}^n (d\pi_g(\wt{X}_i))^2.
\end{align*}
Clearly, for all $f\in C^\infty_c(M)$, $\Delta^G_{\cc H}(f\circ \pi)=\Delta^M_{\cc{H}} f\circ\pi$, that is, Assumption~\ref{SubmersionMeasure} is satisfied. In particular, when $G$ is unimodular, by \cite[Proposition 10, Chapter VII §2]{BourbakiIntegrationII}, $H$ is unimodular as well, and by \cite[Theorem~1]{Lang_book_1974} there is a unique $G$-invariant measure on $M$, denoted by $\mu_M$ such that for any $f\in C^\infty_c(G)$,
\begin{align*}
\int_G f(g) \mu_G(dg)=\int_M f^H(m)\mu_M(dm),
\end{align*}
where $f^H(m)=\int_H f(h\circ q(m))\mu_H(dh)$ for all $m\in M$, $q:M\longrightarrow G$ is a cross section map, and $\mu_G, \mu_H$ are the Haar measures on $G,H$ respectively. The mapping $f\mapsto f^H$ is also surjective from $C^\infty_c(G)$ to $C^\infty_c(M)$. By \cite[Remark~6.18]{DriverGrossSaloff-Coste2010}, the induced sub-Laplacian $\Delta^M_{\cc H}$ is symmetric with respect to $\mu_M$ and its essential self-adjoint extension on $L^2(M,\mu_M)$ generates a self-adjoint Markov semigroup.
\end{example}

We also consider the product of sub-Riemannian manifolds $(M_i,\cc{H}_i, g_i)_{i=1}^n$, where the horizontal distribution on the product space $M=M_1\times\cdots\times M_n$ is given by $\cc{H}=\cc{H}_1\oplus\cdots\oplus\cc{H}_n$ equipped with the sub-Riemannian metric $g=g_1\oplus\cdots\oplus g_n$. It can be easily verified that $(M,\cc{H}, g)$ is a sub-Riemannian manifold. Moreover, the horizontal sub-Laplacian on $(M,\cc H, g)$ is defined by
 
\begin{align}\label{eq:subLaplacian_product}
\Delta^M_{\cc H}=\Delta^{M_1}_{\cc H_1}\oplus\cdots\oplus\Delta^{M_n}_{\cc H_n}.
\end{align} 

The next two theorems show stability of functional inequalities under tensorization and sub-Riemannian submersions of manifolds. In addition to the functional inequalities \eqref{eq:GB}, \eqref{eq:RP} and \eqref{eq:RLS}, we also consider the following \emph{Li-Yau type inequality} on sub-Riemannian manifolds. For all $t>0$ and positive $f\in\cC^\infty_c(M)$

\begin{align}\label{eq:L-Y}\tag{LYI}
C_1(t)\|\nabla_{\cc H} \log P^M_t f\|^2_{\cc H} &\leqslant \frac{\Delta_{\cc H}P^M_t f}{P^M_t f}+C_2(t), 
\end{align}
where $C_1(t), C_2(t)$ are positive, possibly time-dependent constants.

\begin{theorem}\label{thm:tensorization}
	Let $(M_i, \cc{H}_i,g_i, \Delta_{\cc H_i}^{M_i}, \mu_{M_i})$, $i=1, ..., n$ be sub-Riemannian manifolds such that \eqref{eq:GB} (resp. \eqref{eq:RP}, \eqref{eq:RLS}) holds for each horizontal heat semigroup $(P^{(i)}_t)_{t\geqslant 0}$ on $M_i$ with a constant $C_i(t)$. Then the horizontal heat semigroup $(P^M_t)_{t\geqslant 0}$ on $M=M_1\times\cdots\times M_n$ generated by $\Delta^M_{\cc H}$ satisfies \eqref{eq:GB} (resp. \eqref{eq:RP}, \eqref{eq:RLS}) with the constant $C(t)=\max\{C_i(t): 1\leqslant i\leqslant n\}$.
	
	Moreover, if \eqref{eq:L-Y} holds for each $(P^{(i)}_t)_{t\geqslant 0}$ with constants $C^{(i)}_1(t)$ and $C^{(i)}_2(t)$, then $(P^M_t)_{t\geqslant 0}$ satisfies \eqref{eq:L-Y} as well with $C_1(t)=\min\{C^{(i)}_1(t): 1\leqslant i\leqslant n\}$ and $C_2(t)=\sum_{i=1}^n C^{(i)}_2(t)$.
\end{theorem}
\begin{theorem}\label{thm:projection} 
Let $(M, \cc{H}_M, g_M,\Delta_{\cc H}^M, \mu_M)$, $(N, \cc{H}_N, g_N,\Delta_{\cc H}^N, \mu_N)$ be two sub-Riemannian manifolds such that \eqref{eq:GB} (resp. \eqref{eq:RP}, \eqref{eq:RLS},\eqref{eq:L-Y}) holds for $M$ with constant $C(t)$. Assume that there exists a sub-Riemannian submersion from $M$ to $N$ satisfying Assumption~\ref{SubmersionMeasure}. 
Then \eqref{eq:GB} (resp. \eqref{eq:RP}, \eqref{eq:RLS},\eqref{eq:L-Y}) also holds for $N$ with the same constants.
\end{theorem}
In the following section we prove some simple results on the Carnot-Carath\'eodory metric which will be useful in the proof of the above theorems. If we think of such sub-Riemannian manifolds as a Dirichlet space, we can treat such tensor products as  tensor products of Dirichlet forms, see \cite[Section~2.1, Proposition~2.1.2]{BouleauHirschDirichletFormsBook}.

\subsection{Carnot-Carath\'eodory metric on manifolds obtained by tensorization and submersions}
The following lemma states that the Carnot-Carath\'eodory metric is compatible with both tensorization and sub-Riemannian submersions. 
\begin{lemma}\label{lem:metrics}
\begin{enumerate}[leftmargin=*]
\item \label{it:tensor_metric} 
Let $(M_i, \cc{H}_i, g_i)_{i=1}^n$ be a collection of sub-Riemannian manifolds with horizontal distributions $(\cc{H}_i)_{i=1}^n$. If $d_i$ denotes the  Carnot-Carath\'eodory metric on $M_i$, then the Carnot-Carath\'eodory metric on $M_1\times \cdots\times M_n$ with horizontal distribution $(\cc{H}_1\oplus \cdots\oplus \cc{H}_n, g_1\oplus\cdots\oplus g_n)$ is given by
\begin{align}
d(x,y)=\left(\sum_{i=1}^n d_i(x_i, y_i)^2\right)^{\frac{1}{2}}.
\end{align}
\item \label{it:proj_metric} Let $(M, \cc{H}_M, g_M),(N,\cc{H}_N, g_N)$ be two sub-Riemannian manifolds with  Carnot-Carath\'eodory metrics $d_M, d_N$ respectively. Assume that $\pi: M\to N$ is a sub-Riemannian submersion. Then, for any $x,y\in N$, there exists $a\in \pi^{-1}(x), b\in \pi^{-1}(y)$ such that
		\begin{align*}
			d_M(a,b)= d_N(x, y).
		\end{align*}
	\end{enumerate}
\end{lemma}

\begin{proof}
We prove item \eqref{it:tensor_metric} for $n=2$ and the general case will follow by induction. Consider any horizontal curve $\sigma: [0,1]\to M_1\times M_2$ such that $\sigma(0)=x, \sigma(1)=y$. Writing $\sigma=(\sigma_1, \sigma_2)$, it is evident that $\sigma^{\prime}_i(t)\in \cc{H}_i(\sigma_i(t))$ and $\sigma_i(0)=x_i, \sigma_i(1)=y_i$ for $i=1,2$. Therefore, by orthogonality of the horizontal distributions, we have for all $t\in [0,1]$
\begin{align*}
\|\sigma^{\prime}(t)\|^2_{\cc H}= \|\sigma^{\prime}_1(t)\|^2_{\cc{H}_1}+\|\sigma^{\prime}_2(t)\|^2_{\cc H_2}.
\end{align*}
Let us write $l_1=\int_0^1\|\sigma^{\prime}_1(t)\|_{\cc H_1}dt, l_2=\int_0^1\|\sigma^{\prime}_2(t)\|_{\cc H_2}dt$. Then we have
\begin{align*}
&\int_0^1\|\sigma^{\prime}(t)\|_{\cc H} dt
\\
&= \int_0^1\sqrt{\|\sigma^{\prime}_1(t)\|^2_{\cc H_1}+\|\sigma^{\prime}_2(t)\|^2_{\cc H_2}} dt 
\\
&\geqslant \frac{l_1}{\sqrt{l^2_1+l^2_2}}\int_0^1 \|\sigma^{\prime}_1(t)\|_{\cc H_1} dt+\frac{l_2}{\sqrt{l^2_1+l^2_2}}\int_0^1 \|\sigma^{\prime}_2(t)\|_{\cc H_2} dt 
\\
&=\sqrt{l^2_1+l^2_2},
\end{align*}
where the second inequality follows from convexity of the square function.
This shows that $d(x, y) \geqslant \sqrt{d^2_1(x_1, y_1)+d^2_2(x_2, y_2)}$. For the inequality in the opposite direction, we consider unit speed horizontal geodesics $\gamma_1,\gamma_2$ on $M_1, M_2$ respectively such that $\gamma_i(0)=x_i, \gamma_i(d_i(x_i,y_i))=y_i$ for $i=1,2$. The existence of such geodesics is guaranteed by Assumption~\ref{completeness}. Let us define $\sigma_i: [0,1]\to M_i$ such that $\sigma_i(t)=\gamma_i(t d_i(x_i,y_i))$. Then for $\sigma:=(\sigma_1, \sigma_2)$, we have
\begin{align*}
\int_0^1 \|\sigma'(t)\|_{\cc H} dt=\int_0^1\sqrt{\|\sigma'_1(t)\|^2_{\cc H_1}+\|\sigma'_2(t)\|^2_{\cc H_2}} dt=\sqrt{d^2_1(x_1, y_1)+d^2_2(x_2,y_2)}.
\end{align*}
This shows that $d(x,y)\leqslant \sqrt{d^2_1(x_1, y_1)+d^2_2(x_2, y_2)}$ and the proof of item \eqref{it:tensor_metric} is complete.
	
To prove item \eqref{it:proj_metric}, we first claim that for any $a, b\in M$, $d_M(a, b)\geqslant d_N(\pi(a), \pi(b))$. Indeed, for any horizontal curve $\sigma:[0,1]\to M$ with $\sigma(0)=a, \sigma(1)=b$, $\pi\circ\sigma$ defines a horizontal curve on $N$ such that $\pi\circ\sigma(0)=\pi(a), \pi\circ\sigma(1)=\pi(b)$. Moreover, for any $t\in [0,1]$, $\|(\pi\circ \sigma)^{\prime}(t)\|_{\cc H_N}=\|d\pi_{\sigma(t)}(\sigma^{\prime}(t))\|_{\cc H_N}=\|\sigma^{\prime}(t)\|_{\cc H_M}$. Thus 

\begin{align*}
d_M(a,b)&=\inf\left\{\int_0^1 \|(\pi\circ\sigma)^{\prime}(t)\|_{\cc H_N} dt: \sigma(0)=a, \sigma(1)=b, (\pi\circ\sigma)^{\prime}\in\cc{H}_N\right\}
\\
& \geqslant d_N(\pi(a), \pi(b)).
\end{align*}
Now let $x,y\in N$. Then, any absolutely continuous horizontal curve $\gamma:[0,1]\to N$ with $\gamma(0)=x, \gamma(1)=y$ possesses a horizontal lift $\widehat{\gamma}:[0,1]\to M$ such that $\pi\circ\widehat{\gamma}=\gamma$. Let us choose $\gamma$ such that $d_N(x,y)=\int_0^1\|\gamma^{\prime}(t)\|_{\cc H_N}dt$. This implies that 
\begin{align*}
	d_N(x,y)\leqslant d_M(\widehat{\gamma}(0), \widehat{\gamma}(1))\leqslant \int_0^1 \|\widehat{\gamma}^{\prime}(t)\|_{\cc H_M} dt=\int_0^1 \|\gamma^{\prime}(t)\|_{\cc H_N} dt=d_N(x,y).
\end{align*}
Choosing $a=\widehat{\gamma}(0), b=\widehat{\gamma}(1)$ completes the proof of item \eqref{it:proj_metric}.
\end{proof}
\subsection{Proof of Theorem~\ref{thm:tensorization}} 
Stability of \eqref{eq:GB}, \eqref{eq:RP} and \eqref{eq:RLS} under tensorization follows by Theorem~\ref{thm:main} and \eqref{it:tensor_metric} in Lemma~\ref{lem:metrics}. To prove the stability of \eqref{eq:L-Y}, we observe that for any $f\in\cC^\infty_c(M_1\times\cdots\times M_n)$, $t\geqslant 0$, and $1 \leqslant i\leqslant n$,
\begin{align}\label{eq:LY_i}
	C^{(i)}_1(t)\|\nabla_{\cc H_i} \log P^M_t f\|^2_{\cc H_i} &\leqslant \frac{\Delta^{M_i}_{\cc H_i}P^M_t f}{P^M_t f}+C^{(i)}_2(t).
\end{align}
Now, \eqref{eq:subLaplacian_product} implies that $\sum_{i=1}^n\Delta^{M_i}_{\cc H_i} P^M_t f=\Delta^M_{\cc H} P^M_t f$, while the orthogonality of the horizontal distributions $(\cc H_i)_{i=1}^n$ implies 
\begin{align*}
	\sum_{i=1}^n \|\nabla_{\cc H_i} \log P^M_t f\|^2_{\cc H_i}=\|\nabla_{\cc H} \log P^M_t f\|^2_{\cc H}.
\end{align*}
Therefore, the proof of \eqref{eq:L-Y} is concluded by adding the inequalities in \eqref{eq:LY_i} for each $1\le i\le n$. 

\subsection{Proof of Theorem~\ref{thm:projection}}
Denoting the horizontal heat semigroups on $M, N$ by $P^M$ and $P^N$ respectively, we first observe that Assumption~\ref{SubmersionMeasure} implies that for all $f\in \cC^\infty_c(N)$ and $t\geqslant 0$
\begin{align}\label{eq:semigroup_proj}
P^M_t(f\circ \pi)=P^N_tf \circ \pi.
\end{align}
Since $\pi$ is a sub-Riemannian submersion, then for any $x\in M$, $d\pi_x:\cc{H}_M\to \cc{H}_N$ is an isometry, which similarly to \cite[Equation~(4.3)]{GordinaLuo2022} and  \cite[Equation~(3.4)]{GordinaLuo2024} leads to 
\begin{align}\label{eq:gg}
\|\nabla^M_{\cc H}(f\circ \pi)\|^2_{\cc H_M}=\|\nabla^N_{\cc H} f\|_{\cc H_N}^2\circ\pi,
\end{align}
where $\nabla^M_{\cc H}$ (resp. $\nabla^N_\cc{H}$) denotes the horizontal gradient on $(M,\cc H_M)$ (resp. $N, \cc{H}_N$).
Now assume that \eqref{eq:GB} holds for some $p\geqslant 1$. Then for any $t\geqslant 0$ and $f\in \cC^\infty_c(N)$, using \eqref{eq:gg} and \eqref{eq:semigroup_proj} we deduce 
\begin{align*}
	\|\nabla^N_{\cc H}P^N_t f\|^p_{\cc H_N}\circ\pi&=\|\nabla^M_{\cc H}(P^N_tf\circ\pi)\|^p_{\cc H_M}\\
	&=\|\nabla^M_{\cc H} P^M_t(f\circ \pi)\|^p_{\cc H_M} \\
	&\leqslant C(p,t) P^M_t \|\nabla^M_{\cc H} (f\circ\pi)\|_{\cc H_M} \\
	&=C(p,t)P^M_t(\|\nabla^N_{\cc H} f\|^p_{\cc H_N}\circ\pi) \\
	&=C(p,t) P^N_t\|\nabla^N_{\cc H} f\|^p_{\cc H_N}\circ \pi.
\end{align*}
Since $\pi$ is surjective, \eqref{eq:GB} follows for $P^N_t$. Stability of \eqref{eq:RP}, \eqref{eq:RLS} and \eqref{eq:L-Y} under the mapping $\pi$ is a direct consequence of \eqref{eq:semigroup_proj}.

\section{Examples}\label{sec:4}

\subsection{Kolmogorov diffusion on Euclidean spaces}
 
We consider a Kolmogorov-type diffusion operator on $\R^d\times \R^d$  given by 
\begin{align}
	Lf(x,y)=\sum_{j=1}^d x_j\frac{\partial f}{\partial y_j}(x,y)+\sum_{j=1}^d \sigma^2_j \frac{\partial^2 f}{\partial x^2_j}(x,y),
\end{align}
where $\sigma^2_j>0$ for all $j=1,\ldots, d$.

\begin{proposition}
For all $t\geqslant 0$ and $f\in \cC^1_b(\R^d\times \R^d)$ we have
\begin{align}
& |\nabla_x P_t f|^p \leqslant P_t\left(\sum_{i=1}^d\left|\frac{\partial f}{\partial x_i}+t\frac{\partial f}{\partial y_i}\right|^2\right)^{\frac{p}{2}} \text{ for all } p\geqslant 1,
\label{eq:GB-K} 
\\
& \sum_{i=1}^d \left(\frac{\partial P_tf}{\partial x_i}-\frac{1}{2}t\frac{\partial P_t f}{\partial y_i}\right)^2+\frac{t^2}{12}\left(\frac{\partial P_t f}{\partial y_i}\right)^2  \leqslant \frac{1}{\sigma^2 t}(P_t f^2 - (P_t f)^2), 
\label{eq:RP-K}
\end{align}
\begin{equation}\label{eq:RLS-K}
\begin{aligned}
& \sum_{i=1}^d \left(\frac{\partial \log P_tf}{\partial x_i}-\frac{t}{2}\frac{\partial\log P_t f}{\partial y_i}\right)^2+\frac{t^2}{12}\left(\frac{\partial \log P_t f}{\partial y_i}\right)^2 
\\
& \leqslant \frac{1}{\sigma^2 t P_t f}(P_t (f\log f) - P_t f \log(P_t f)),
\end{aligned}
\end{equation}
where $\sigma^2:=\min\{\sigma^2_j: 1\leqslant j\leqslant d\}$. 
\end{proposition}

\begin{remark}
We note that in \cite{BaudoinGordinaMariano2020}, the authors considered $\sigma^2_j=\sigma^2$ for all $j=1,\ldots, d$. The above result proves similar gradient estimates for the non-isotropic case as well, that is, when $\sigma^2_j$ are not identical.
\end{remark}
\begin{proof}
We first observe that by Jensen's inequality
\begin{align*}
	\left(P_t\left(\sum_{i=1}^d \left|\frac{\partial f}{\partial x_i}+t\frac{\partial f}{\partial y_i}\right|^2\right)^{\frac{1}{2}}\right)^p\le  P_t\left(\sum_{i=1}^d\left|\frac{\partial f}{\partial x_i}+t\frac{\partial f}{\partial y_i}\right|^2\right)^{\frac{p}{2}},
\end{align*}
it suffices to prove \eqref{eq:GB-K} for $p=1$. We use an idea similar to the proof of \cite[Theorem~2]{Zhang2023}. For each $i=1,\ldots, d$, let $P^{(i)}$ denote the semigroup generated by 
	\begin{align*}
	L_i=x_i\frac{\partial}{\partial y_i}+\sigma^2_i \frac{\partial^2}{\partial x_i^2}.
	\end{align*}
	 Then for each $t>0$, $P_t=P^{(1)}_t\otimes\cdots\otimes P^{(d)}_t$. Now for $d=1$, \cite[Proposition 2.10]{BaudoinGordinaMariano2020} implies that for any $i=1,\ldots, d$, 
\begin{align}\label{eq:GB-K-1}
|\nabla_{x_i}P^{(i)}_t f|\leqslant P^{(i)}_t\left(\left|\frac{\partial f}{\partial x_i}+t\frac{\partial f}{\partial y_i}\right|\right).
\end{align}
Let $a_1,\ldots, a_d\in\bb R$ be such that $\sum_{i=1}^d a^2_i=1$. Then using \eqref{eq:GB-K-1} followed by the Cauchy-Schwarz inequality and monotonicity of $P_t$, we have 
\begin{align*}
	\sum_{i=1}^d a_i |\nabla_{x_i} P_t f|\leqslant P_t\left(\sum_{i=1}^d a_i \left|\frac{\partial f}{\partial x_i}+t\frac{\partial f}{\partial y_i}\right|\right)\leqslant P_t\left(\sum_{i=1}^d \left|\frac{\partial f}{\partial x_i}+t\frac{\partial f}{\partial y_i}\right|^2\right)^{\frac{1}{2}}.
\end{align*}
Optimizing with respect to $a_1,\ldots, a_d$ yields the inequality in \eqref{eq:GB-K} for $p=1$. 
 Next, to prove \eqref{eq:RP-K} and \eqref{eq:RLS-K}, it is known from \cite{BaudoinGordinaMariano2020} that for any $d\geqslant 1$, the control distance associated to the squared gradient $\Gamma_t(f)=\sum_{i=1}^d(\frac{\partial f}{\partial x_i}-\frac{t}{2}\frac{\partial f}{\partial y_i})^2+\frac{t^2}{12}\sum_{i=1}^d(\frac{\partial f}{\partial y_i})^2$ on $\bb{R}^d\times \bb{R}^d$ is given by  
	\begin{align*}
		d_t((x_1,y_1),(x_2,y_2))^2=4|x_1-x_2|^2+\frac{12}{t}\langle x_1-x_2,y_1-y_2\rangle+\frac{12}{t^2}|y_1-y_2|^2.
	\end{align*}
As a result, for any $f\in\cC^1(\bb R^d\times \bb{R}^d)$, 
\begin{align*}
\Gamma_t(f)(x)=\lim_{r\to 0}\sup_{y:d_t(x,y)\leqslant r} \left|\frac{f(x)-f(y)}{d_t(x,y)}\right|.
\end{align*}
Again, invoking \cite[Proposition 2.7, 2.8]{BaudoinGordinaMariano2020} for $d=1$, both \eqref{eq:RP-K} and \eqref{eq:RLS-K} follow from Theorem~\ref{thm:main}.
\end{proof}

\subsection{Kinetic Fokker-Planck equations} 

Consider the stochastic differential equation on $\R^d\times \R^d$
\begin{align}
\begin{cases}
dX_t=Y_t dt 
\\
d Y_t= dW_t-\nabla V(X_t)dt-Y_t dt,
\end{cases}
\end{align}
where $\left\{ W_t \right\}_{t\geqslant 0}$ is an $\bb R^d$-valued Brownian motion and $V(x)=\sum_{i=1}^d V_i(x_i)$ with $V_i\geqslant 0$. The semigroup associated with the equation is denoted $P_t$. Under the assumption of polynomial growth of the potentials $V_i$ as in \cite[Theorem~A.8]{Villani2009}, that is, 
\begin{align}\label{eq:villani}
	|V^{\prime \prime}_i(x)|\leqslant C(1+V^{\prime}_i(x)) \text{ for all } 1\leqslant i\leqslant d \text{ and for some constant } C,
\end{align}
we have the following result.

\begin{proposition}
	Assume that \eqref{eq:villani} holds. Then, there exists a dimension-independent constant $c$ such that for all $f\in\cC^1(\bb R^d\times \bb R^d)$ and $t>0$, 
	\begin{align}\label{eq:Fokker_Plank_RPI}
		\|\nabla P_t f\|^2\leqslant \frac{c}{(t\wedge 1)^3}P_t f^2.
	\end{align}
\end{proposition}
The constant $c$ only depends on the constant $C$ in \eqref{eq:villani}. It is worth noting that both in the pointwise bound (see \cite{GuillinWangFY2012}) and the $L^2$ bound (see \cite[Theorem~A.8]{Villani2009}), the constants appearing on the right hand side of the inequalities of the form \eqref{eq:Fokker_Plank_RPI} are dimension-dependent.
\begin{proof}
Let $P^{(i)}$ denote the semigroup associated to the one-dimensional diffusion
\begin{align*}
		dX^{(i)}_t&=dY^{(i)}_t \\
		dY^{(i)}_t&=dW^{(i)}_t-V^{\prime}_i(X_t)dt-Y^{(i)}_tdt.
\end{align*}	
Then $P=P^{(1)}\otimes\cdots\otimes P^{(d)}$. Invoking \cite[Corollary 3.3]{GuillinWangFY2012}, for each $1\leqslant i\leqslant d$ we have 
\begin{align*}
	\|\nabla_{(x_i, y_i)}P^{(i)}_t f\|^2\leqslant \frac{c}{(t\wedge 1)^3}P^{(i)}_t f^2.
\end{align*}
Hence the proposition follows from Theorem~\ref{thm:main}.
\end{proof}

\begin{proposition}
Assume that there exist constants $m,M >0$, such that $\sqrt{M}-\sqrt{m} \leqslant 1$ and
\[
m \leqslant \nabla^2 V \leqslant M.
\]
There exist dimension-independent constants $c_1,c_2>0$ such that for all $f\in\cC^1(\bb R^d\times \bb R^d)$ and $t>0$, 
	\begin{align}
		\|\nabla P_t f\|^2\leqslant c_1 e^{-c_2 t} P_t (\| \nabla f\|^2 )
	\end{align}
\end{proposition}

\begin{proof}
The proposition follows from Theorem~\ref{thm:main} and \cite[Theorem 2.12]{Baudoin2017a}.
\end{proof}

\subsection{Lie groups with transverse symmetries}\label{sec:3}

Denote by $\bG$ a real or complex Lie group, by $\mathfrak{g}\cong T_{e} \bG$ its Lie algebra identified with the tangent space at the identity $e$. For each $A \in \mathfrak{g}$, let $\widetilde{A}$ denote the unique extension of $A$ to a left-invariant vector field on $\bG$. For $n\geqslant 1$ and $1\leqslant m\leqslant n$ we consider a $2n+m$-dimensional Lie group $\bG$ whose Lie algebra $(\mathfrak{g}, \langle\cdot,\cdot \rangle)$ has an orthonormal basis $\{X_1, \ldots, X_{2n}, Z_1, \ldots, Z_m\}$ such that for $1\leqslant l\leqslant m$
\begin{align*}
	[X_i, X_j]=\sum_{l=1}^m A_{l; i,j}Z_l \quad\text{and} \quad [X_i, Z_l]=\sum_{j=1}^{2n} R_{l; i,j} X_j,
\end{align*}
where the family of matrices $\left\{ A_l, R_l \right\}_{l=1}^n$ satisfies the following conditions
\begin{enumerate}[label=(\Alph*)]
	\item $\left\{ A_l, R_l \right\}_{l=1}^m$ are $2n\times 2n$ skew-symmetric matrices.
	\item \label{it:B} $\left\{ A_l \right\}_{l=1}^m$ are linearly independent and of full rank.
	\item $A_lA_k=A_kA_l, \ R_lR_k=R_kR_l, \ A_lR_k=R_kA_l$ for all $1\leqslant l, k\leqslant m$,
\end{enumerate}
and $A_{l;i,j}$ (resp. $R_{l;i,j}$) denotes the $(i,j)^{th}$ entry of $A_l$ (resp. $R_l$).
Note that these assumptions imply that the corresponding left-invariant vector fields $(\wt{X}_i)_{i=1}^{2n}$ define a sub-Riemannian structure on $\bG$ with the horizontal distribution $\cc{H}_{\operatorname{G}}=\Span\{\wt{X}_i: 1\leqslant i\leqslant 2n\}$, and the sub-Riemannian metric $(g_x(\cdot, \cdot))_{x\in \bG}$ given by 
\begin{align*}
	g_x(\wt{X}(x), \wt{Y}(x))=\langle X, Y\rangle \text{ for all } x \in M,\  X, Y \in \cc{H}_{\operatorname{G}}.
\end{align*}
The Lie group $\operatorname{G}$ is equipped with the left-invariant Haar measure and we consider the sub-Laplacian on $\bG$ defined by 
\begin{align*}
\Delta^{\operatorname{G}}_{\cc H}=\sum_{i=1}^{2n} \wt{X}^2_i.
\end{align*}
We note that $\bG$ is a sub-Riemannian manifold with a transverse symmetry as described in \cite{BaudoinGarofalo2017}. 
\begin{example} Let us consider $n=m=1$. Then
	\begin{itemize}[leftmargin=.5cm]
	\item the 3-dimensional Heisenberg group $\bb{H}$ corresponds to the case when $A_1=\begin{pmatrix} 0 & 1 \\ -1 & 0 \end{pmatrix}$, and $R_1=\mathbf{0}_2$.
	\item $\operatorname{SU}(2)$, the group of all $2\times 2$ unitary matrices with determinant equal to $1$, corresponds to the case when $A_1=R_1=\begin{pmatrix} 0 & 1 \\ -1 & 0 \end{pmatrix}$.
	\item $\operatorname{SL}(2)$, the group of all invertible real matrices with determinant equal to $1$, corresponds to the case when $A_1=-R_1=\begin{pmatrix} 0 & 1 \\ -1 & 0 \end{pmatrix}$.
	\end{itemize}
\end{example}

We start by explaining that while the techniques introduced in \cite{BaudoinGarofalo2017} work in this setting, the constants one gets are dimension-dependent. In terms of the notation in the aforementioned paper, it follows that for all $f\in\cC^\infty(\bG)$
\begin{align*}
	\cc{R}(f)&=\sum_{i=1}^{2n}\sum_{j=1}^{2n}\left(\sum_{k=1}^{2n} A_{l;i,k}R_{l; k,j}\right)X_i f X_jf+\frac{1}{2}\sum_{1\leqslant i\leqslant j\leqslant 2n}\left(\sum_{l=1}^m A_{l;i,j}Z_l f\right)^2 
	\\
	\cc{T}(f)&=\sum_{l=1}^m \sum_{i=1}^{2n}\left(\sum_{j=1}^{2n} A_{l;i,j} X_j f\right)^2 =\sum_{l=1}^m\|A_l\nabla_{\cc H}f\|^2_{\cc H}.
\end{align*}
A simple computation shows that for all $f\in\cC^\infty(\bG)$, one has 
\begin{align*}
	\cc{R}(f)&\geqslant \rho\|\nabla_{\cc H} f\|^2_{\cc H}+\gamma\|\nabla_{\cc V} f\|^2_{\cc V}, \\ \cc{T}(f)&\le\kappa\|\nabla_{\cc V}f\|^2_{\cc V},
\end{align*}
where $\rho$ is the minimum eigenvalue of
\begin{align}\label{eq:Lambda}
	\Lambda=\sum_{l=1}^m A_lR_l,
\end{align}
and
\begin{align}\label{eq:rho_kappa}
	\gamma=\frac{1}{2}\inf_{\|x\|=1}\sum_{1\leqslant i\leqslant j\leqslant 2n}\left(\sum_{l=1}^m A_{l;i,j}x_j\right)^2, \quad \kappa=\sup_{\|x\|=1}\sum_{l=1}^m\|A_l x\|^2.
\end{align}
As a result, \cite[Theorem 2.19]{BaudoinGarofalo2017} implies that group $\bG$ satisfies the generalized curvature-dimension inequality $\operatorname{CD}(\rho, \gamma, \kappa, 2n)$ with respect to the sub-Laplacian $\Delta^{\operatorname{G}}_{\cc H}$. However, in general the parameters $\rho,\gamma$ and $\kappa$ depend on the dimension of $\bG$, see Remark~\ref{r.4.2} and Proposition~\ref{prop:optimal_constants} below. 

To this end, we explain our method which leads to gradient estimates for the heat kernel $P^{\operatorname{G}}_t=e^{t\Delta^{\operatorname{G}}_{\cc H}}$ on $\bG$ that do not depend on $\gamma, \kappa$, and $n$. First, we show that the Lie group $\bG$ can be reduced to a product of $3$-dimensional model groups introduced in \cite{BakryBaudoinBonnefont2009}. We note that $\Lambda$ defined in \eqref{eq:Lambda} is a symmetric matrix which is unitary equivalent to a diagonal matrix of the form 
\[D=\begin{pmatrix}
	D_{1} & & \\
	& \ddots & \\
	& & D_{n}
\end{pmatrix}
\]
where $D_i=\rho_iI_{2}$, $I_2$ being the $2\times 2$ identity matrix, see Lemma~\ref{prop:diagonalization} for details. For each $1\leqslant i\leqslant n$, let $\bb{M}(\rho_i)$ denote the $3$-dimensional model space as introduced in \cite{BakryBaudoinBonnefont2009}, that is, $\bb M(\rho_i)$ is a simply connected Lie group  whose Lie algebra is given by $\mathfrak{m}_i=\Span\{X^{\prime}_{2i-1}, X^{\prime}_{2i}, Z'_i \}$ such that 
\begin{align}
	[X^{\prime}_{2i-1}, X^{\prime}_{2i}]=Z^{\prime}_i, \quad [X^{\prime}_{2i-1}, Z^{\prime}_i]=-\rho_i X^{\prime}_{2i}, \quad [X^{\prime}_{2i}, Z^{\prime}_i]=\rho_i X^{\prime}_{2i-1},
\end{align}
and $\left\{X^{\prime}_{2i-1}, X^{\prime}_{2i}, Z'_i \right\}$ forms an orthonormal basis for $\mathfrak{m}_i$. For each $1\le i\le n$, $\bb{M}(\rho_i)$ is equipped with the sub-Riemannian structure where the horizontal distribution is given by $\cc{H}_i=\Span\{X^{\prime}_{2i-1}, X^{\prime}_{2i}\}$. Denoting 
\begin{align}\label{eq:M}
\bb M=\bb M(\rho_1)\times\cdots\times \bb M(\rho_n),
\end{align}
 we observe that $\bb M$ is also a connected Lie group with Lie algebra $\mathfrak{m}=\mathfrak{m}_1\oplus\cdots\oplus\mathfrak{m}_n$. According to the discussion in Subsection~\ref{sec:3.1}, $\bb{M}$ is equipped with a sub-Riemmanian structure where the horizontal distribution is given by $\cc{H}_{\bb M}=\cc{H}_1\oplus\cdots\oplus \cc{H}_n$. Due to the existence of a sub-Riemannian submersion from $\mathbb{M}$ to $\operatorname{G}$, and by  Proposition~\ref{thm:submersion}, we obtain the following dimension-independent inequalities for the horizontal heat semigroup $(P^{\operatorname{G}}_t)_{t\ge 0}$.
\begin{theorem}\label{thm:transverse}
	The horizontal heat semigroup $(P^{\operatorname{G}}_t)_{t\geqslant 0}$ on $\bG$ satisfies \eqref{eq:RP}, \eqref{eq:RLS}, and \eqref{eq:L-Y}. More precisely, one has for all $t>0$ and $f\in\cC^\infty_c(\operatorname{G})$,
	
	\begin{align}
		\|\nabla_{\cc H} P^{\operatorname{G}}_t f\|^2_{\cc H} & \leqslant \frac{5+\rho^- t}{2t}(P^{\operatorname{G}}_tf^2-(P^{\operatorname{G}}_t f)^2) \label{eq:RP-Tr} 
	\end{align}
	and when $f>0$,
	\begin{align}
		P^{\operatorname{G}}_t f\|\nabla_{\cc H}\log P^{\operatorname{G}}_t f\|^2_{\cc H} & \leqslant \frac{5+\rho^- t}{t} \left(P^{\operatorname{G}}_t(f\log f)-P^{\operatorname{G}}_t f\log(P^{\operatorname{G}}_t f)\right), \label{eq:RLS-Tr}  \\
		\|\nabla_{\cc H}\log P^{\operatorname{G}}_t f\|^2_{\cc H}&\leqslant \left(4-\frac{2\rho t}{3}\right)\frac{\Delta^{\operatorname{G}}_{\cc H}P^{\operatorname{G}}_t f}{P^{\operatorname{G}}_t f}+\frac{n\rho^2}{3}t-4n\rho+\frac{16n}{t}, \label{eq:li_yau} 
	\end{align}
	where $\Lambda$ is defined in \eqref{eq:Lambda}, $\rho$ is the minimum eigenvalue of $\Lambda$, and $\rho^-=\max\{0,-\rho\}$.
	Additionally, when $\Lambda$ is non-negative definite, we have the following results. 
	\begin{enumerate}[leftmargin=*]
		\item \label{it:thm_G_2a} $(P^{\operatorname{G}}_t)_{t\geqslant 0}$ satisfies \eqref{eq:GB}, that is, for all $p>1$, $f\in \cC^\infty_c(\bG)$ and $t>0$, 
		\begin{align}\label{eq:GB-Tr}
			\|\nabla_{\cc H} P^{\operatorname{G}}_t f\|_{\cc H}\leqslant C_p e^{-t\rho}\left(P^{\operatorname{G}}_t \|\nabla_{\cc H} f\|^p\right)^{\frac{1}{p}},
		\end{align}
		In fact, one has 
		\begin{align}\label{eq:GB-Tr2}
			\|\nabla_{\cc H}P^{\operatorname{G}}_t f\|^2_{\cc H}+\|\nabla_{\cc V}P^{\operatorname{G}}_t f\|^2_{\cc V}\leqslant C(P^{\operatorname{G}}_t\|\nabla_{\cc H} f\|^2_{\cc H}+P^{\operatorname{G}}_t\|\nabla_{\cc V} f\|^2_{\cc V})
		\end{align}
		for some positive constant $C$ independent of $n$.
		\item \label{it:thm_G_2b} 
		For all $0<s<t$ and nonnegative function $f\in\cC_b(\bG)$ we have
		\begin{align}\label{eq:PHI}
			P^{\operatorname{G}}_t f(x)\leqslant P^{\operatorname{G}}_t f(y) \left(\frac{t}{s}\right)^{8n}\exp\left(\frac{4d^2(x,y)}{t-s}\right). \tag{PHI} 
		\end{align}
	\end{enumerate}
\end{theorem}

\begin{remark}\label{r.4.2}
	We note that \eqref{eq:RP}, \eqref{eq:RLS} and \eqref{eq:L-Y} for $\bG$ could also be obtained using the generalized curvature-dimension inequality using the results in \cite{BaudoinBonnefont2012}. More precisely, \cite[Proposition~3.1 and 3.2]{BaudoinBonnefont2012} shows that the generalized curvature-dimension inequality $\operatorname{CD}(\rho,\gamma,\kappa,\infty)$ implies 
	\begin{align}
		tP^{\operatorname{G}}_t\|\nabla_{\cc H}\log P^{\operatorname{G}}_t f\|_{\cc H}^2&\leqslant \left(1+\frac{2\kappa}{\gamma}+\rho^{-}\right) (P^{\operatorname{G}}_t(f\log f)-(P^{\operatorname{G}}_t f)(\log P^{\operatorname{G}}_t f)) 
		\\
		t\|\nabla_{\cc H} P^{\operatorname{G}}_t f\|_{\cc H}^2&\leqslant \frac{1}{2}\left(1+\frac{2\kappa}{\gamma}+\rho^{-}\right)(P^{\operatorname{G}}_t f^2-(P^{\operatorname{G}}_t f)^2),
	\end{align}
	and \cite[Theorem 6.1]{BaudoinGarofalo2017} implies the Li-Yau estimate 
	\begin{align*}
		& \|\nabla_{\cc H}\log P^{\operatorname{G}}_t f\|^2_{\cc H}
		\\
		& \leqslant \left(1+\frac{3\kappa}{2\gamma}-\frac{2\rho t}{3}\right)\frac{\Delta_{\cc H}P^{\operatorname{G}}_t f}{P^{\operatorname{G}}_t f}+\frac{n\rho^2 t}{3}-n\rho\left(1+\frac{3\kappa}{2\gamma}\right)+\frac{n(1+\frac{3\kappa}{2\gamma})^2}{t}.
	\end{align*}
	In the next proposition we show that the constants in Theorem~\ref{thm:transverse} are sharper. 
\end{remark}

\begin{proposition}\label{prop:optimal_constants}
	Let $\rho,\gamma$ and $\kappa$ be as in \eqref{eq:rho_kappa}. Then for any $n\geqslant 1$ and $m=n$, $\frac{\kappa}{\gamma}\geqslant 2$. Moreover, for any $1\leqslant m\leqslant n$, one can choose a family of skew-symmetric commuting matrices $\left\{ A_l \right\}_{1\leqslant l\leqslant m}$ such that $\frac{\kappa}{\gamma}=\frac{2(n^m-1)}{n(n-1)}$. 	
\end{proposition}
To prove the above results, we first observe that the vector fields $\{\wt{X}_i\}_{i=1}^{2n}$ can be \emph{decoupled} without changing the Lie algebra structure and the sub-Laplacian.

\begin{lemma}\label{prop:diagonalization}
	Without loss of generality we can assume that 
	\begin{equation}\label{eq:A_R}
		\begin{aligned}
		R_l&=\mathbf{0}_{2n-2r}\oplus \begin{pmatrix}0 & -\mu_{1l} \\
				\mu_{1l} & 0 \end{pmatrix}\oplus\cdots\oplus \begin{pmatrix} 0 & -\mu_{rl} \\ \mu_{rl} & 0 \end{pmatrix} \\
				A_l&=\begin{pmatrix}0 & -\lambda_{1l} 
				\\
				\lambda_{1l} & 0 \end{pmatrix}\oplus\cdots\oplus 
			\begin{pmatrix} 0 & -\lambda_{nl} \\ \lambda_{nl} & 0 \end{pmatrix}
		\end{aligned}
	\end{equation}
	for some $0 \leqslant r\leqslant n$ and real numbers $\lambda_{il}, \mu_{il}$. Moreover, for each $1\leqslant l\leqslant m$, there exist $1\leqslant i\leqslant n,1\leqslant j\leqslant r$ such that $\lambda_{il}\not= 0, \mu_{jl}\not= 0$.
\end{lemma}
\begin{remark}\label{rem:full_rank}
	We note that the linear independence condition in \ref{it:B} equivalent to the full column rank of the matrix $(\lambda_{il})_{1\le i\le n, 1\le l \le m}$.
\end{remark}
\begin{proof}
	Observe that $(A_l, R_l)_{1\leqslant l\leqslant m}$ is a family of normal matrices. As a result, \cite[Theorem~2.5.15]{HornJohnsonBook2013} implies that there exists a $2n\times 2n$ orthogonal matrix $U$ such that for all $1\leqslant l\leqslant m$
	\begin{align*}
		UA_l U^\top&=\begin{pmatrix}0 & -\lambda_{1l} 
			\\
			\lambda_{1l} & 0 \end{pmatrix}\oplus\cdots\oplus \begin{pmatrix} 0 & -\lambda_{nl} 
			\\ 
			\lambda_{nl} & 0 \end{pmatrix} 
		\\
		UR_l U^\top&=\mathbf{0}_{2n-2r}\oplus 
		\begin{pmatrix}0 & \mu_{1l} \\
			-\mu_{1l} & 0 \end{pmatrix}\oplus\cdots\oplus \begin{pmatrix} 0 & \mu_{rl} \\ -\mu_{rl} & 0 \end{pmatrix},
	\end{align*}
where $\mathbf{0}_k$ denotes the zero matrix of dimension $k\times k$.
	Writing $U=(U_{ij})_{1\leqslant i,j\leqslant 2n}$, let us now define 
	\begin{align*}
		Y_i=\sum_{j=1}^{2n} U_{ij}X_j.
	\end{align*}
	After a simple computation, it follows that 
	\begin{align}
		[Y_i, Y_j]=
		\begin{cases}
			\sum_{l=1}^m\lambda_{sl}Z_l & \mbox{if $i=2s-1, j=2s, 1\le s\leqslant n$} 
			\\
			0 & \text{ otherwise.}
		\end{cases}
	\end{align}
	and 
	\begin{align}
		[Y_i, Z_l]=
		\begin{cases}
			0 & \mbox{if $1\leqslant i\leqslant 2r$} 
			\\
			-\mu_{sl} Y_{2s} & \mbox{if $i=2s-1, r<s\leqslant n$} 
			\\
			\mu_{sl} Y_{2s-1} & \mbox{if $i=2s$, $r<s\leqslant n$} 
			\\
			0 & \mbox{otherwise}
		\end{cases}
	\end{align}
	Moreover, $\{\wt{Y}_i\}_{i=1}^{2n}$ is also a left-invariant orthonormal frame in $\bG$. Therefore from \cite[Theorem~3.6]{GordinaLaetsch2016} one gets
	\begin{align*}
		\Delta^{\operatorname{G}}_{\cc H}=\sum_{i=1}^{2n} \wt{X}^2_i=\sum_{i=1}^{2n} \wt{Y}^2_i.
	\end{align*}
	This proves the lemma.
\end{proof}
\begin{proposition}\label{thm:submersion}	Let $\bb{M}$ be defined as in \eqref{eq:M}.
	Then there exists a sub-Riemannian submersion $\Phi: \bb M\to\operatorname{G}$.
\end{proposition}
\begin{proof} 
	We note that for each $1\leqslant i\leqslant n$, 
	$\rho_i=\sum_{l=1}^m \lambda_{il} \mu_{il}$ is an eigenvalue of $\Lambda$.
	Writing $\mathfrak{m}_i=\Span\{X^{\prime}_{2i-1}, X^{\prime}_{2i}, Z^{\prime}_i\}$ as the Lie algebra of $\bb{M}_i$, we define the following linear map $\phi:\oplus_{i=1}^n \mathfrak{m}_i\to \mathfrak{g}$ such that
	\begin{align*}
		\phi(X^{\prime}_{2i-1})={X_{2i-1}}, \quad \phi(X^{\prime}_{2i})=X_{2i}, \quad \phi(Z^{\prime}_i)=\sum_{l=1}^m {\lambda_{il}} Z_l.
	\end{align*}
From Remark~\ref{rem:full_rank} it follows that $\phi$ is a surjective linear map.
	We note that for any $1\leqslant i\leqslant n$ and $1\leqslant l\leqslant m$ one gets
	\begin{equation}
		\begin{aligned}
			\phi([X^{\prime}_{2i-1}, X^{\prime}_{2i}])&=\phi(Z^{\prime}_i)\\
			&=\sum_{l=1}^m\lambda_{il}Z_l \\
			&=[X_{2i-1}, X_{2i}] \\
			&=[\phi(X^{\prime}_{2i-1}), \phi(X^{\prime}_{2i})].
		\end{aligned}
	\end{equation}
	Also
	\begin{equation}
		\begin{aligned}
			\phi([X^{\prime}_{2i-1}, Z^{\prime}_i])&=-\rho_i\phi(X^{\prime}_{2i})\\
			&=-\rho_i X_{2i}\\
			&=-\sum_{l=1}^m \lambda_{il}\mu_{il} X_{2i}\\
			&=[\phi(X^{\prime}_{2i-1}), \phi(Z^{\prime}_i)].
		\end{aligned}
	\end{equation}
	This shows that $\phi:\oplus_{i=1}^n\mathfrak{m}_i\to \mathfrak{g}$ is a surjective Lie algebra homomorphism. Since $\bb M$ is also connected, there exists a surjective Lie group homomorphism $\Phi:\bb M\to\operatorname{G}$ such that $\phi=(d\Phi)_e$, $e$ being the identity element of $\bb M$. Moreover, $\phi$ maps an orthonormal basis of $\cc{H}_{\mathbb M}$ to the same of $\cc{H}_\bG$, which proves that $\Phi$ is a sub-Riemannian submersion.
\end{proof}
In the next lemma, we show that the model spaces $\bb{M}(\rho_i)$ can be replaced by $\bb H$, $\operatorname{SU}(2)$ or $\wt{\operatorname{SL}}(2)$ depending on $\rho_i$, where $\wt{\operatorname{SL}}(2)$ is the universal cover of $\operatorname{SL}(2)$. This lemma will be useful to prove \eqref{eq:GB-Tr}. Let $P^{(i)}$ denote the horizontal heat semigroup generated by $\Delta_i=(X^{\prime}_{2i-1})^2+(X^{\prime}_{2i})^2$ on $\bb{M}(\rho_i)$.
\begin{lemma}\label{lem:M(rho)}
	For each $1\leqslant i\leqslant n$, there exists a Lie group isomorphism $\pi_i: \bb{M}(\rho_i)\to H_i$, where 
	\begin{align*}
		H_i=\begin{cases}
			\bsu(2) & \mbox{if $\rho_i> 0$} \\
			\wt{\bsl}(2) & \mbox{if $\rho_i<0$} \\
			\bb{H} & \mbox{if $\rho_i=0$}
		\end{cases}
	\end{align*}
	such that 
	$P^{(i)}_t (f\circ \pi_i)= Q^{(i)}_{t\alpha_i} f \circ \pi_i$ for all $t\geqslant 0$ and $f\in \cC^\infty_c(H_i)$, where $Q^{(i)}$ is the horizontal heat semigroup on $H_i$, and $\alpha_i=\mathbbm{1}_{\{\rho_i=0\}}+|\rho_i|\mathbbm{1}_{\{\rho_i\neq 0\}}$.
\end{lemma}
\begin{proof}
	When $\rho_i=0$, it is easy to see that $\mathfrak{m}_i$ is isometrically isomorphic to the Lie algebra of $\bb H$. 
	For $\rho_i\neq 0$, let us define $Y_{2i-1}=\frac{X^{\prime}_{2i-1}}{\sqrt{|\rho_i|}}, Y_{2i}=\sign(\rho_i)\frac{X^{\prime}_{2i}}{\sqrt{|\rho_i|}}, W_i=\frac{Z^{\prime}_i}{\rho_i}$. Then one has
	\begin{align}
		[Y_{2i-1}, Y_{2i}]=W_i, \ [Y_{2i-1}, W_i]=-\sign(\rho_i) Y_{2i}, \ [Y_{2i}, W_i]=\sign(\rho_i) Y_{2i-1}.
	\end{align}
	Considering $\left\{\frac{X^{\prime}_{2i-1}}{\sqrt{|\rho_i|}},\sign(\rho_i)\frac{X^{\prime}_{2i}}{\sqrt{|\rho_i|}}\right\}$ as an orthonormal basis of $\mathfrak{su}(2)$ or $\wt{\mathfrak{sl}}(2)$, the lemma follows from the proof of Theorem~\ref{thm:projection}.
\end{proof}

\begin{proof}[Proof of Theorem~\ref{thm:transverse}]
	
	We first show \eqref{eq:RP-Tr} and \eqref{eq:RLS-Tr}. By Theorem~\ref{thm:tensorization}, Theorem~\ref{thm:projection} and Theorem~\ref{thm:submersion}, it suffices to show the validity of \eqref{eq:RP}, \eqref{eq:RLS} for each $\bb{M}(\rho_i)$, which follows from the generalized curvature-dimension inequality proved in \cite[Proposition~2.1]{BaudoinGarofalo2017} in conjunction with \cite[Proposition 3.1, 3.2]{BaudoinBonnefont2012}. For \eqref{eq:li_yau}, using \cite[Remark 6.2]{BaudoinGarofalo2017}, we note that for any $i$, $\bb{M}(\rho_i)$ satisfies $\operatorname{CD}(\rho,\frac{1}{2}, \frac{1}{2})$. As a result, \eqref{eq:li_yau} follows from \cite[Theorem~6.1, Eq. (6.1)]{BaudoinGarofalo2017} and Theorem~\ref{thm:tensorization}, Theorem~\ref{thm:projection}. To prove \eqref{it:thm_G_2a}, we note that the nonnegativity of $\Lambda$ enforces that $\rho_i\geqslant 0$ for each $i=1,\ldots, d$. As a result, Lemma~\ref{lem:M(rho)} implies that $\bb{M}(\rho_i)\cong\bb{H} \text{ or} \ \operatorname{SU}(2)$ according to $\rho_i=0$ or $\rho_i>0$. Using Driver-Melcher inequality \cite{DriverMelcher2005} or H.~Q.~Li inequality \cite{LiHong-Quan2006} for $\rho_i=0$ and Lemma~\ref{lem:M(rho)} together with \cite[Theorem 4.10]{BaudoinBonnefont2009} for $\rho_i>0$, it follows that for all $\rho_i\geqslant 0$, $p>1$ and $f\in \cC^\infty_c(\bb{M}(\rho_i))$

	\begin{align}
		\|\nabla^{\bb{M}(\rho_i)}_{\cc H} P^{\bb{M}(\rho_i)}_t f\|^p_{\cc H}\leqslant C^{(i)}_pe^{-\rho_i t} P^{\bb{M}(\rho_i)}_t\|\nabla^{\bb{M}(\rho_i)}_{\cc H} f\|^p_{\cc H}.
	\end{align}
	Therefore \eqref{eq:GB-Tr} follows from Theorem~\ref{thm:tensorization} and Theorem~\ref{thm:projection} as well. For \eqref{eq:GB-Tr2}, due to \cite[Lemma~2.10]{BaudoinGarofalo2017}, we have  $[\Delta^{{\rm G}}_{\cc H}, Z_l]=0$ for any $1\leqslant l\leqslant m$. Therefore, for any $t\geqslant 0$, $\nabla_{\cc V}P^{\operatorname{G}}_t=P^{\operatorname{G}}_t\nabla_{\cc V}$, which entails that for all $f\in \cC^\infty_c(\bG)$, 
	\begin{align*}
		\|\nabla_{\cc V} P^{\operatorname{G}}_t f\|^2\leqslant P^{\operatorname{G}}_t\|\nabla_{\cc V} f\|^2,
	\end{align*}
	proving \eqref{eq:GB-Tr2}. To prove \eqref{it:thm_G_2b}, we again resort to \cite[Theorem~7.1]{BaudoinGarofalo2017} to argue that for any $1\leqslant i\leqslant n$, $0<s<t$ and nonnegative $f\in\cC_b(\bb{M}(\rho_i))$,
	\begin{align*}
		P^{(i)}_s f(x_i)\leqslant P^{(i)}_t f(y_i) \left(\frac{t}{s}\right)^{8}\exp\left(\frac{4 d^2_i(x_i, y_i)}{t-s}\right), \quad \forall x_i, y_i\in\bb{M}(\rho_i),
	\end{align*}
	where $d_i$ denotes the Carnot-Carath\'eodory metric on $\bb{M}(\rho_i)$.
	As a result, \eqref{eq:PHI} follows from an argument very similar to the proof of Wang-Harnack inequality in Theorem~\ref{thm:main} accompanied with Lemma~\ref{lem:metrics} and Theorem~\ref{thm:submersion}.
\end{proof}

\begin{proof}[Proof of Proposition~\ref{prop:optimal_constants}]
	
	As noted in Lemma~\ref{prop:diagonalization}, without loss of generality we can assume that
	\begin{align*}
		A_l=\begin{pmatrix}0 & -\lambda_{1l} \\
			\lambda_{1l} & 0 \end{pmatrix}\oplus\cdots\oplus \begin{pmatrix} 0 & -\lambda_{nl} \\ \lambda_{nl} & 0 \end{pmatrix}.
	\end{align*}
	Therefore, $\kappa$ and $\gamma$ becomes, respectively,
	\begin{align*}		\kappa=\sup_{\|\bx\|=1}\sum_{j=1}^n\sum_{l=1}^m\lambda_{jl}^2(x^2_{2j-1}+x^2_{2j})=\max\left\{1\leqslant j\leqslant n: \sum_{l=1}^m \lambda^2_{jl}\right\},
	\end{align*}
	\begin{align*}
		\gamma=\frac{1}{2}\inf_{\|\bz\|=1}\sum_{i=1}^n \sum_{l=1}^m\lambda_{il}^2z^2_l=\min\left\{1\leqslant l\leqslant m: \sum_{i=1}^n \lambda^2_{il}\right\}.
	\end{align*}
	When $m=n$, the above computation implies that $\kappa\geqslant 2\gamma$, which proves the first statement of the proposition.
	
	Let us now choose $\lambda_{il}=i^{l-1}$, $1\leqslant i\leqslant n, 1\leqslant l\leqslant m$. Note that this choice of $\lambda_{il}$ ensures that $A_1,\ldots, A_m$ are linearly independent. Now, with this choice of $(A_l)_{1\leqslant l\leqslant m}$, we have
	\begin{align*}
		& \kappa=\max\left\{1\leqslant j\leqslant n: \sum_{l=1}^m j^{l-1}\right\}=\frac{n^{m}-1}{n-1}, 
		\\
		& \gamma=\frac{1}{2}\min\left\{1\leqslant l\leqslant m: \sum_{i=1}^n i^{l-1}\right\}=\frac{n}{2}.
	\end{align*}
	As a result, $\kappa/\gamma=\frac{2(n^m-1)}{n(n-1)}$, which completes the proof of the proposition. 
\end{proof}

The next example is a special case where we assume $R_l=\mathbf{0}_{2n}$ for all $1\le l\le m$.
\subsubsection{Step $2$ homogeneous Carnot group}
Consider the following step $2$ homogeneous Carnot group $\GG_{n,m}=\bb{R}^{2n}\times \bb{R}^m$ with the group operation
\begin{align}
	(\bx, \by)\star (\bx',\by')=(\bx+\bx', y_1+\langle  A_1\bx, \bx'\rangle,\ldots, y_m+\langle A_m\bx, \bx'\rangle),
\end{align}
where $\bx\in\bb{R}^{2n}$, $\by\in\bb{R}^m$, and $(A_l)_{1\leqslant l\leqslant m}$ is a family of commuting, linearly independent skew-symmetric real matrices. The linear independence of $(A_l)_{1\leqslant l\leqslant m}$ implies that the Lie algebra of $\GG_{n,m}$ is generated by the left invariant vector fields $\{X_1, \ldots, X_{2n}\}$ where 
\begin{align}\label{eq:vf-H}
	X_i=\frac{\partial}{\partial x_i}-\frac{1}{2}\sum_{l=1}^m \langle A_l\bx, \mathbf{e}_i\rangle\frac{\partial}{\partial y_l}.
\end{align}
After some simple computations, it follows that for any $1\leqslant i\neq j\leqslant 2n$, 
\begin{align*}
	[X_i, X_j]=\sum_{l=1}^n A_{l; i,j}\frac{\partial}{\partial y_l},
\end{align*}
where $A_{l; i,j}$ denotes the $(i,j)^{th}$ entry of $A_l$. Also, $[X_i, \frac{\partial}{\partial y_l}]=0$ for all $1\leqslant i\leqslant 2n$ and $1\leqslant l\leqslant m$.
In particular, when $m=1$ and 
\begin{align*}
A_1=\begin{pmatrix}
	S & & \\
	 &  \ddots & \\
	& & S
\end{pmatrix}, \quad  S=\begin{pmatrix}0 & 1 \\ -1 & 0 \end{pmatrix},
\end{align*}
then 
 $\bb{G}_{n,1}=\bb{H}_{2n+1}$, the Heisenberg group of dimension $2n+1$.

\begin{proposition}\label{prop:Heisenberg}
	For any $p\geqslant 1$, There exists a positive constant $C$ such that for all $t>0$ and $f\in \cC^\infty_c(\bb{G}_{n,m})$
	\begin{align}
		\|\nabla_{\cc H} P^{\bb G_{n,m}}_t f\|_{\cc{H}}&\leqslant C P^{\bb G_{n,m}}_t\|\nabla_{\cc H} f\|_{\cc H}  \label{eq:GB-H}  \\
		\|\nabla_{\cc H} P^{\bb G_{n,m}}_t f\|^2_{\cc H}&\leqslant \frac{1}{t}(P^{\bb G_{n,m}}_t f^2-(P^{\bb G_{n,m}}_t f)^2), \label{eq:RP-H} 
	\end{align} and for all $f\in \cC^\infty_c(\bb{G}_{n,m})$ with $f>0$,
	\begin{align}
		P^{\bb G_{n,m}}_t f \|\nabla_{\cc H} P^{\bb G_{n,m}}_t f\|^2_{\cc H}&\leqslant \frac{5}{t} (P^{\bb G_{n,m}}_t (f\ln f)- P^{\bb G_{n,m}}_t f\ln (P^{\bb G_{n,m}}_t f)) \label{eq:RLS-H} \\
		\|\nabla_{\cc H} \ln P^{\bb G_{n,m}}_t f\|^2_{\cc H}& \leqslant 4\frac{\Delta_{\cc H}P^{\bb G_{n,m}}_t f}{P^{\bb G_{n,m}}_t f}+\frac{16n}{t}\label{eq:LY-H}
	\end{align}
\end{proposition}

\begin{remark}
We note that \eqref{eq:RP-H} for such groups has already been studied in \cite[Theorem~2.1]{WangFY2016a} using a derivative type formula for a class of diffusion Markov semigroups. However, the constants obtained there grow with the dimension of the group while our results show that the constants can be chosen independent of the dimension. This can be used to extend our results to  infinite dimensions. 
\end{remark}

\begin{remark}
When $\bb{G}_{n,m}=\bb{H}_{2n+1}$, the Heisenberg group of dimension $2n+1$, Baudoin and Bonnefont \cite[Corollary 2.7]{BaudoinBonnefont2016} showed that \eqref{eq:RP} holds with $C(t)=\frac{n+1}{2nt}$. Clearly, this estimate is sharper than \eqref{eq:RP-H} when $n\geqslant 2$. In general, for any two-step Carnot group, the optimal choice for $C(t)$ is bounded above by $\frac{2n+2m}{2}=n+m$, see \cite[Proposition~2.6]{BaudoinBonnefont2016}, which grows with the dimension of the group.
\end{remark}

\begin{remark}
 When $m=1$, $\bb{G}_{n,m}$ is a non-isotropic Heisenberg group, for which the gradient bound \eqref{eq:GB-H} has been obtained in \cite{Zhang2023, LiHong-QuanZhang2019}. See also \cite{GordinaLuo2022, GordinaLuo2024} for logarithmic Sobolev  inequalities on the non-isotropic Heisenberg groups.
\end{remark}

\begin{proof}[Proof of Proposition~\ref{prop:Heisenberg}]
	In terms of the notation in Theorem~\ref{thm:transverse}, we note that $\Lambda=0$. Therefore, \eqref{eq:GB-H}, \eqref{eq:RLS-H} and \eqref{eq:LY-H} follows from Theorem~\ref{thm:transverse}. To prove \eqref{eq:RP-H}, we note that Theorem~\ref{thm:submersion} implies that there exists a sub-Riemannian submersion from $\bb{H}^n$ to $\bb{G}_{n,m}$. Also, the optimal reverse Poincar\'e inequality for the Heisenberg group $\bb H$ has been proved in \cite[Corollary~2.7]{BaudoinBonnefont2016}. As a result, \eqref{eq:RP-H} is a direct consequence of Theorem~\ref{thm:tensorization} and Theorem~\ref{thm:projection}. 
\end{proof}

\subsection{Hypoelliptic heat equation on $\operatorname{SO}(3)$}\label{s.4.3} The Lie group $\operatorname{SO}(3)$ is the group of $3\times 3$ real orthogonal matrices of determinant $1$. A basis of the Lie algebra $\mathfrak{so}(3)$ is $\{X_1,X_2,Z\}$, where
\begin{align*}
	X_1=\begin{pmatrix}
		0 & 1 & 0\\
		-1 & 0 & 0\\
		0 & 0 & 0
	\end{pmatrix}, \ \ 
	X_2=\begin{pmatrix}
		0 & 0 & 1\\
		0 & 0 & 0\\
		-1 & 0 & 0
	\end{pmatrix}, \ \ 
	Z=\begin{pmatrix}
		0 & 0 & 0\\
		0  & 0 & 1\\
		0 & -1 & 0
	\end{pmatrix}.
\end{align*}
They satisfy the following commutation rules
\begin{equation}\label{brackets} 
	[X_1,X_2]=-Z, \quad [X_1,Z]= X_2, \quad [X_2,Z]=-X_1.
\end{equation}
Let $\wt X_1,\wt X_2$ denote the left-invariant vector fields corresponding to $X_1,X_2$. Then, $\wt{X}_1,\wt{X}_2$ satisfy the H\"{o}rmander's condition and therefore $\operatorname{SO}(3)$ is a sub-Riemannian manifold. We consider $\operatorname{SO}(3)$ equipped with the Haar measure and the sub-Laplacian is defined by 
\begin{align*}
	\Delta^{\operatorname{SO}(3)}_{\cc H}=\wt{X}^2_1+\wt{X}^2_2.
\end{align*}
Using the notations from Section~\ref{sec:3}, we note that $A_1=R_1=\begin{pmatrix} 0 & -1 \\ 1 & 0 \end{pmatrix}$. As a consequence, we have the following result.
\begin{corollary}
	The horizontal heat semigroup $P^{\operatorname{SO}(3)}_t$ generated by $\Delta^{\operatorname{SO}(3)}_\cc{H}$ satisfies \eqref{eq:RP-Tr}, \eqref{eq:RLS-Tr}, \eqref{eq:li_yau} and \eqref{eq:GB-Tr} with $\rho=1$.
\end{corollary}

\subsection{Hypoelliptic heat equation on $\operatorname{SO}(4)$}
The Lie algebra of $\operatorname{SO}(4)$ is given by
\begin{align*}
	\mathfrak{so}(4)=\{A\in \mathfrak{gl}_4(\bb R): \ A+A^\top=0\},
\end{align*}
endowed with the inner product $\langle A,B\rangle_{\mathfrak{so}(4)}=\trace(AB^\top)$. Also, $\operatorname{SO}(4)$ is equipped with the Haar measure. Let $E_{i,j}\in\mathfrak{gl}_4(\bb R)$ denote the matrix whose $(i,j)^{th}$ entry equals $1$ and the rest of the entries are equal to $0$. We consider 
\begin{align*}
X_j=E_{j+1,1}-E_{1,j+1}, \quad 1 \leqslant j \leqslant 3,
\end{align*}
and let us write 
\begin{align*}
Z_1=[X_2,X_3], \ Z_2=[X_3,X_1], \ Z_3=[X_1,X_2].
\end{align*}
Then, a simple computation shows that 
\begin{equation}
\begin{aligned}
&[Z_1, Z_2]=Z_3, \quad [Z_2, Z_3]=Z_1, \quad [Z_3, Z_1]=Z_2, \\
&[Z_i, X_j]=\epsilon_{ijk} X_k,
\end{aligned}
\end{equation}
where 
\begin{align*}
\epsilon_{ijk}=\begin{cases} 
	0 & \text{ if } i=j  
\\
	1 & \text{ if } (i,j,k) \text{ is an even permutation of }  (1,2,3)
\\
	-1 & \text{ otherwise }
\end{cases}
\end{align*}
 Moreover, $\{X_1, X_2, X_3, Z_1, Z_2, Z_3\}$ is an orthonormal basis for $\mathfrak{so}(4)$. Writing the corresponding left-invariant vector fields by $\wt{X}_i, \wt{Z}_i$ for $i=1,2,3$, the Laplace-Beltrami operator on $\operatorname{SO}(4)$ is given by 
 \begin{align*}
 	\Delta^{\operatorname{SO}(4)}=\wt{X}^2_1+\wt{X}^2_2+\wt{X}^2_3+\wt{Z}^2_1+\wt{Z}^2_2+\wt{Z}^2_3.
 \end{align*}
Let us consider the horizontal distribution $\cc{H}_{\operatorname{SO}(4)}$ on $\operatorname{SO}(4)$ equipped with the inner product 
\begin{align*}
	g_x(\wt{X},\wt{Y})=\trace(XY^\top),
\end{align*}
and generated by the orthonormal frame $\{\wt{X}_1, \wt{X}_2, \wt{Z}_1, \wt{Z}_2\}$. In this case, the sub-Laplacian is given by 
\begin{align*}
\Delta^{\operatorname{SO}(4)}_{\cc H}=\wt{X}^2_1+\wt{X}^2_2+\wt{Z}^2_1+\wt{Z}^2_2
\end{align*}
With a relabeling of the vector fields $X_1, X_2, Z_1, Z_2, X_3, Z_3$ by $X'_1$, $X'_2$, $X'_3$, $X'_4$, $Z'_1$, $Z'_2$ respectively, we note that 
\begin{align*}
	&[X'_1, X'_2]=Z'_2, \quad [X'_1, X'_3]=0, \quad [X'_1, X'_4]=Z'_1 \\
	&[X'_2, X'_3]=-Z'_1, \quad [X'_2, X'_4]=0, \quad [X'_3, X'_4]=Z'_2,
\end{align*}
and 
\begin{align*}
	&[X'_1, Z'_1]=-X'_4, \quad [X'_2, Z'_1]=X'_3, \quad [X'_3, Z'_1]=-X'_2, \quad [X'_4, Z'_1]=X'_1 \\
	&[X'_1, Z'_2]=-X'_2, \quad [X'_2, Z'_2]=X'_1, \quad [X'_3, Z'_2]=-X'_4, \quad [X'_4, Z'_2]=X'_3.
\end{align*}
From the above computation it follows that $\operatorname{SO}(4)$ is a Lie group with transverse symmetry as defined in Subsection~\ref{sec:3}. In this case, the matrices $\{A_l,R_l\}_{l=1}^2$ are given by 
\begin{align*}
	&A_1=\begin{pmatrix}
		0 & 0 & 0 & 1 \\
		0 & 0 & -1 & 0 \\
		0 & 1 & 0 & 0 \\
		-1 & 0 & 0 & 0
	\end{pmatrix} , \quad
	A_2=\begin{pmatrix}
		0 & 1 & 0 & 0 \\
		-1 & 0 & 0 & 0 \\
		0 & 0 & 0 & 1 \\
		0 & 0 & -1 & 0
	\end{pmatrix}, \quad 
\\
&R_1=\begin{pmatrix}
	0 & 0 & 0 & -1 \\
	0 & 0 & 1 & 0 \\
	0 & -1 & 0 & 0 \\
	1 & 0 & 0 & 0
\end{pmatrix}, \quad
R_2=\begin{pmatrix}
	0 & -1 & 0 & 0 \\
	1 & 0 & 0 & 0 \\
	0 & 0 & 0 & -1 \\
	0 & 0 & 1 & 0
\end{pmatrix}.
\end{align*}
Also, it can be easily verified that $A_1, A_2, R_1, R_2$ commute with each other.
Let $\left(P^{\operatorname{SO}(4)}_t\right)_{t\geqslant 0}$ denote the horizontal heat semigroup on $\operatorname{SO}(4)$ generated by $\Delta^{\operatorname{SO}(4)}_{\cc H}$. 

\begin{corollary}
The horizontal heat semigroup $P^{\operatorname{SO}(4)}_t$ satisfies \eqref{eq:RP-Tr}, \eqref{eq:RLS-Tr}, \eqref{eq:li_yau} and \eqref{eq:GB-Tr} with $\rho=2$.
\end{corollary}
\begin{proof}
	For $\operatorname{SO}(4)$, the matrix $\Lambda$ defined in \eqref{eq:Lambda} is given by $\Lambda=A_1R_1+A_2R_2=2I_{4\times 4}$, where $I_{4\times 4}$ is the $4\times 4$ identity matrix. Therefore, $\rho=2$, which proves the result.
\end{proof}

\subsection{Compact Heisenberg nil-manifolds}
	Let $G=\bb{H}$, the $3$--dimensional Heisenberg group and $H=G\cap \bb{Z}^3$. Then, $M:=H\backslash G$ is a compact \emph{Heisenberg nil-manifold}. We refer to \cite{Folland2004a, glutsyuk2024} for more details about the definition. Moreover, by \cite{Vinberg1988}, if $K$ is any closed subgroup of $\bb{H}$ such that $K\backslash \bb{H}$ is a compact manifold, then there exists $k\in\bb{N}$ such that
	\begin{align*}
		K=\{(x,y,z/k)\in\bb{H}: x,y, z\in \bb{Z}\}.
	\end{align*}
Noting that $H$ is unimodular with the counting measure as the Haar measure, for any $f\in C^\infty_c(G)$
\begin{align*}
	f^H(m)=\sum_{h\in H} f(h\circ q(m)) \in C^\infty_c(M) \text{ for any } m\in M,
\end{align*}
where $q:M\to G$ is the cross section map.
Let $\mu_M$ be the $G$-invariant measure on $M$ inherited from the Haar measure on $G$ as described in Example~\ref{example:homogeneous_space}. Then by \cite[Equation~(6.23)]{DriverGrossSaloff-Coste2010}, for any $f\in C^\infty_c(M)$ and $g\in G$ we have 
\begin{align*}
	\Delta^M_{\cc H} (f^H)(Hg)=\sum_{h\in H}\Delta^G_{\cc H} f(h\circ g).
\end{align*}
Moreover, from the discussion in Example~\ref{example:homogeneous_space}, $\mu_M$ is symmetrizing for $\Delta^M_{\cc H}$, and therefore by Theorem~\ref{thm:projection} we get the following result.
\begin{corollary}
	Let $M=H\backslash G$ be the Heisenberg nil-manifold equipped with the sub-Riemannian structure induced by the natural projection map $\pi:G\longrightarrow H\backslash G$. Then \eqref{eq:GB} (resp.~\eqref{eq:RP}, \eqref{eq:RLS}, \eqref{eq:L-Y}) holds for $(M, \cc{H}_M, \langle\cdot,\cdot\rangle_M, \mu_M)$ with the same constants as in Proposition~\ref{prop:Heisenberg} with $n=1$.
\end{corollary}

\subsection{Grushin plane} Let $G=\bb{H}$, the $3$-dimensional Heisenberg group and $H=\{(0,y,0): y\in\bb{R}\}$ be the closed subgroup of $G$. Then, the Grushin plane can be defined as the homogeneous space $M=H\backslash G$. In this case, $M$ can be identified with $\bb{R}^2$ and the projection map $\pi: G\longrightarrow M$ is given by
\begin{align*}
	\pi(g)=\pi(x,y,z)=(x, z+\frac{1}{2}xy).
\end{align*}
Let 
\begin{align*}
\wt{X}=\frac{\partial}{\partial x}-\frac{y}{2}\frac{\partial}{\partial z}, \quad \wt{Y}=\frac{\partial}{\partial y}+\frac{x}{2}\frac{\partial}{\partial z}
\end{align*}
be the orthonormal frame in $\cc{H}_G$. Then, according to \cite[Section~4.2]{GordinaLuo2024}, $d\pi_g(\wt X)=\frac{\partial}{\partial x}$ and $d\pi_g(\wt Y)=x\frac{\partial}{\partial y}$. Moreover, with respect to the induced sub-Riemannian structure on $M$, the sub-Laplacian is given by 
 \begin{align*}
 	\Delta^M_{\cc H}=\frac{\partial^2}{\partial x^2}+x^2\frac{\partial^2}{\partial y^2},
 \end{align*}
and the measure $\mu_M$ in Example~\ref{example:homogeneous_space} is the Lebesgue measure on $\R^2$, see \cite[p.~464]{DriverGrossSaloff-Coste2009a}.
As a consequence of Theorem~\ref{thm:projection}, we have the following result.
\begin{corollary}
\eqref{eq:GB} (resp. \eqref{eq:RP}, \eqref{eq:RLS}, \eqref{eq:L-Y}) holds for the sub-Riemannian manifold $(M,\cc{H}_M, \langle\cdot, \cdot\rangle_{\cc H_M}, \mu_M)$ with the same constants as in Proposition~\ref{prop:Heisenberg} with $n=1$.
\end{corollary}

	\providecommand{\bysame}{\leavevmode\hbox to3em{\hrulefill}\thinspace}
\providecommand{\MR}{\relax\ifhmode\unskip\space\fi MR }
\providecommand{\MRhref}[2]{%
  \href{http://www.ams.org/mathscinet-getitem?mr=#1}{#2}
}
\providecommand{\href}[2]{#2}


\begin{thebibliography}{10}

\bibitem{AgrachevBarilariBoscainBook2020}
Andrei Agrachev, Davide Barilari, and Ugo Boscain, \emph{A comprehensive
  introduction to sub-{R}iemannian geometry}, Cambridge Studies in Advanced
  Mathematics, vol. 181, Cambridge University Press, Cambridge, 2020, From the
  Hamiltonian viewpoint, With an appendix by Igor Zelenko. \MR{3971262}

\bibitem{Agrachev}
Andrei Agrachev, Ugo Boscain, Jean-Paul Gauthier, and Francesco Rossi,
  \emph{The intrinsic hypoelliptic {L}aplacian and its heat kernel on
  unimodular {L}ie groups}, J. Funct. Anal. \textbf{256} (2009), no.~8,
  2621--2655. \MR{2502528}

\bibitem{AmbrosioGigliSavareBook2005}
Luigi Ambrosio, Nicola Gigli, and Giuseppe Savar\'e, \emph{Gradient flows in
  metric spaces and in the space of probability measures}, Lectures in
  Mathematics ETH Z\"urich, Birkh\"auser Verlag, Basel, 2005. \MR{2129498}

\bibitem{BakryEmery1985}
D.~Bakry and M.~{\'E}mery, \emph{Diffusions hypercontractives}, S\'eminaire de
  probabilit\'es, {XIX}, 1983/84, Lecture Notes in Math., vol. 1123, Springer,
  Berlin, 1985, pp.~177--206. \MR{889476 (88j:60131)}

\bibitem{BakryBaudoinBonnefont2009}
Dominique Bakry, Fabrice Baudoin, Michel Bonnefont, and Bin Qian,
  \emph{Subelliptic {L}i-{Y}au estimates on three dimensional model spaces},
  Potential theory and stochastics in {A}lbac, Theta Ser. Adv. Math., vol.~11,
  Theta, Bucharest, 2009, pp.~1--10. \MR{2681833 (2012b:58038)}

\bibitem{BanerjeeGordinaMariano2018}
Sayan Banerjee, Maria Gordina, and Phanuel Mariano, \emph{Coupling in the
  heisenberg group and its applications to gradient estimates}, Ann. Probab.
  \textbf{46} (2018), no.~6, 3275--3312.

\bibitem{BarilariRizzi2013}
Davide Barilari and Luca Rizzi, \emph{A formula for {P}opp's volume in
  sub-{R}iemannian geometry}, Anal. Geom. Metr. Spaces \textbf{1} (2013),
  42--57. \MR{3108867}

\bibitem{Baudoin2017a}
Fabrice Baudoin, \emph{Bakry-\'{E}mery meet {V}illani}, J. Funct. Anal.
  \textbf{273} (2017), no.~7, 2275--2291. \MR{3677826}

\bibitem{Baudoin2022}
\bysame, \emph{Geometric inequalities on {R}iemannian and sub-{R}iemannian
  manifolds by heat semigroups techniques}, New trends on analysis and geometry
  in metric spaces, Lecture Notes in Math., vol. 2296, Springer, Cham, [2022]
  \copyright 2022, pp.~7--91. \MR{4432545}

\bibitem{BaudoinBonnefont2009}
Fabrice Baudoin and Michel Bonnefont, \emph{The subelliptic heat kernel on
  {${\rm SU}(2)$}: representations, asymptotics and gradient bounds}, Math. Z.
  \textbf{263} (2009), no.~3, 647--672. \MR{2545862 (2011d:58060)}

\bibitem{BaudoinBonnefont2012}
\bysame, \emph{Log-{S}obolev inequalities for subelliptic operators satisfying
  a generalized curvature dimension inequality}, J. Funct. Anal. \textbf{262}
  (2012), no.~6, 2646--2676. \MR{2885961}

\bibitem{BaudoinBonnefont2016}
\bysame, \emph{Reverse {P}oincar\'e inequalities, isoperimetry, and {R}iesz
  transforms in {C}arnot groups}, Nonlinear Anal. \textbf{131} (2016), 48--59.
  \MR{3427969}

\bibitem{BaudoinDemniWangBook2024}
Fabrice Baudoin, Nizar Demni, and Jing Wang, \emph{Stochastic areas, horizontal
  {B}rownian motions, and hypoelliptic heat kernels}, EMS Tracts in
  Mathematics, vol.~37, EMS Press, Berlin, [2024] \copyright 2024. \MR{4809935}

\bibitem{BaudoinEldredge2021}
Fabrice Baudoin and Nathaniel Eldredge, \emph{Transportation inequalities for
  {M}arkov kernels and their applications}, Electron. J. Probab. \textbf{26}
  (2021), Paper No. 45, 30. \MR{4244339}

\bibitem{BaudoinGarofalo2017}
Fabrice Baudoin and Nicola Garofalo, \emph{Curvature-dimension inequalities and
  {R}icci lower bounds for sub-{R}iemannian manifolds with transverse
  symmetries}, J. Eur. Math. Soc. (JEMS) \textbf{19} (2017), no.~1, 151--219.
  \MR{3584561}

\bibitem{BaudoinGordinaHerzog2021}
Fabrice Baudoin, Maria Gordina, and David~P. Herzog, \emph{Gamma {C}alculus
  {B}eyond {V}illani and {E}xplicit {C}onvergence {E}stimates for {L}angevin
  {D}ynamics with {S}ingular {P}otentials}, Arch. Ration. Mech. Anal.
  \textbf{241} (2021), no.~2, 765--804. \MR{4275746}

\bibitem{BaudoinGordinaMariano2020}
Fabrice Baudoin, Maria Gordina, and Phanuel Mariano, \emph{Gradient bounds for
  {K}olmogorov type diffusions}, Ann. Inst. Henri Poincar\'{e} Probab. Stat.
  \textbf{56} (2020), no.~1, 612--636. \MR{4059002}

\bibitem{BaudoinGordinaMelcher2013}
Fabrice Baudoin, Maria Gordina, and Tai Melcher, \emph{Quasi-invariance for
  heat kernel measures on sub-{R}iemannian infinite-dimensional {H}eisenberg
  groups}, Trans. Amer. Math. Soc. \textbf{365} (2013), no.~8, 4313--4350.
  \MR{3055697}

\bibitem{BouleauHirschDirichletFormsBook}
Nicolas Bouleau and Francis Hirsch, \emph{Dirichlet forms and analysis on
  {W}iener space}, de Gruyter Studies in Mathematics, vol.~14, Walter de
  Gruyter \& Co., Berlin, 1991. \MR{1133391 (93e:60107)}

\bibitem{BourbakiIntegrationII}
Nicolas Bourbaki, \emph{Integration. {II}. {C}hapters 7--9}, Elements of
  Mathematics (Berlin), Springer-Verlag, Berlin, 2004, Translated from the 1963
  and 1969 French originals by Sterling K. Berberian. \MR{2098271}

\bibitem{CamrudGordinaHerzogStoltz2022}
Evan Camrud, David~P. Herzog, Gabriel Stoltz, and Maria Gordina, \emph{Weighted
  {$L^2$}-contractivity of {L}angevin dynamics with singular potentials},
  Nonlinearity \textbf{35} (2022), no.~2, 998--1035. \MR{4373993}

\bibitem{Cheeger1999}
J.~Cheeger, \emph{Differentiability of {L}ipschitz functions on metric measure
  spaces}, Geom. Funct. Anal. \textbf{9} (1999), no.~3, 428--517. \MR{1708448}

\bibitem{CoulhonJiangKoskelaSikora2020}
Thierry Coulhon, Renjin Jiang, Pekka Koskela, and Adam Sikora, \emph{Gradient
  estimates for heat kernels and harmonic functions}, J. Funct. Anal.
  \textbf{278} (2020), no.~8, 108398, 67. \MR{4056992}

\bibitem{DriverGrossSaloff-Coste2009a}
Bruce~K. Driver, Leonard Gross, and Laurent Saloff-Coste, \emph{Holomorphic
  functions and subelliptic heat kernels over {L}ie groups}, J. Eur. Math. Soc.
  (JEMS) \textbf{11} (2009), no.~5, 941--978. \MR{2538496 (2010h:32052)}

\bibitem{DriverGrossSaloff-Coste2010}
\bysame, \emph{Growth of {T}aylor coefficients over complex homogeneous
  spaces}, Tohoku Math. J. (2) \textbf{62} (2010), no.~3, 427--474.
  \MR{2742018}

\bibitem{DriverMelcher2005}
Bruce~K. Driver and Tai Melcher, \emph{Hypoelliptic heat kernel inequalities on
  the {H}eisenberg group}, J. Funct. Anal. \textbf{221} (2005), 340--365.

\bibitem{EldredgeGordinaSaloff-Coste2018}
Nathaniel Eldredge, Maria Gordina, and Laurent Saloff-Coste,
  \emph{Left-invariant geometries on $\operatorname{SU}(2)$ are uniformly
  doubling}, Geometric and Functional Analysis \textbf{28} (2018), no.~5,
  1321--1367.

\bibitem{FeffermanPhong1983}
C.~Fefferman and D.~H. Phong, \emph{Subelliptic eigenvalue problems},
  Conference on harmonic analysis in honor of {A}ntoni {Z}ygmund, {V}ol. {I},
  {II} ({C}hicago, {I}ll., 1981), Wadsworth Math. Ser., Wadsworth, Belmont, CA,
  1983, pp.~590--606. \MR{730094}

\bibitem{Folland2004a}
G.~B. Folland, \emph{Compact {H}eisenberg manifolds as {CR} manifolds}, J.
  Geom. Anal. \textbf{14} (2004), no.~3, 521--532. \MR{2077163}

\bibitem{glutsyuk2024}
A.~Glutsyuk and Yu. Sachkov, \emph{Sub-riemannian geodesics on the heisenberg
  3d nil-manifold}, 2024.

\bibitem{GoldbergIshihara1978}
S.~I. Goldberg and T.~Ishihara, \emph{Riemannian submersions commuting with the
  {L}aplacian}, J. Differential Geometry \textbf{13} (1978), no.~1, 139--144.
  \MR{520606}

\bibitem{Gordina2017}
Maria Gordina, \emph{An application of a functional inequality to
  quasi-invariance in infinite dimensions}, pp.~251--266, Springer New York,
  New York, NY, 2017.

\bibitem{GordinaLaetsch2016}
Maria Gordina and Thomas Laetsch, \emph{Sub-{L}aplacians on {s}ub-{R}iemannian
  {m}anifolds}, Potential Anal. \textbf{44} (2016), no.~4, 811--837.
  \MR{3490551}

\bibitem{GordinaLaetsch2017}
\bysame, \emph{A convergence to {B}rownian motion on sub-{R}iemannian
  manifolds}, Trans. Amer. Math. Soc. \textbf{369} (2017), no.~9, 6263--6278,
  In print: September 2017. \MR{3660220}

\bibitem{GordinaLuo2022}
Maria Gordina and Liangbing Luo, \emph{Logarithmic {S}obolev inequalities on
  non-isotropic {H}eisenberg groups}, J. Funct. Anal. \textbf{283} (2022),
  no.~2, Paper No. 109500. \MR{4410358}

\bibitem{GordinaLuo2024}
\bysame, \emph{{L}ogarithmic {S}obolev inequalities on homogeneous spaces},
  International Mathematics Research Notices (2024), rnae205.

\bibitem{GromovBook2007}
Misha Gromov, \emph{Metric structures for {R}iemannian and non-{R}iemannian
  spaces}, english ed., Modern Birkh\"{a}user Classics, Birkh\"{a}user Boston,
  Inc., Boston, MA, 2007, Based on the 1981 French original, With appendices by
  M. Katz, P. Pansu and S. Semmes, Translated from the French by Sean Michael
  Bates. \MR{2307192}

\bibitem{GuillinWangFY2012}
Arnaud Guillin and Feng-Yu Wang, \emph{Degenerate {F}okker-{P}lanck equations:
  {B}ismut formula, gradient estimate and {H}arnack inequality}, J.
  Differential Equations \textbf{253} (2012), no.~1, 20--40. \MR{2917400}

\bibitem{HeinonenKoskelaShanmugalingamTysonBook2015}
Juha Heinonen, Pekka Koskela, Nageswari Shanmugalingam, and Jeremy~T. Tyson,
  \emph{Sobolev spaces on metric measure spaces}, New Mathematical Monographs,
  vol.~27, Cambridge University Press, Cambridge, 2015, An approach based on
  upper gradients. \MR{3363168}

\bibitem{Hormander1967a}
Lars H{\"o}rmander, \emph{Hypoelliptic second order differential equations},
  Acta Math. \textbf{119} (1967), 147--171. \MR{0222474 (36 \#5526)}

\bibitem{HornJohnsonBook2013}
Roger~A. Horn and Charles~R. Johnson, \emph{Matrix analysis}, second ed.,
  Cambridge University Press, Cambridge, 2013. \MR{2978290}

\bibitem{JerisonSanchez-Calle1986}
David~S. Jerison and Antonio S\'{a}nchez-Calle, \emph{Estimates for the heat
  kernel for a sum of squares of vector fields}, Indiana Univ. Math. J.
  \textbf{35} (1986), no.~4, 835--854. \MR{865430}

\bibitem{Kuwada2010a}
Kazumasa Kuwada, \emph{Duality on gradient estimates and {W}asserstein
  controls}, J. Funct. Anal. \textbf{258} (2010), no.~11, 3758--3774.
  \MR{2606871 (2011d:35109)}

\bibitem{Kuwada2013b}
\bysame, \emph{Gradient estimate for {M}arkov kernels, {W}asserstein control
  and {H}opf-{L}ax formula}, Potential theory and its related fields, RIMS
  K\^{o}ky\^{u}roku Bessatsu, B43, Res. Inst. Math. Sci. (RIMS), Kyoto, 2013,
  pp.~61--80. \MR{3220453}

\bibitem{Lang_book_1974}
Serge Lang, \emph{{${\rm SL}\sb{2}({\bf R})$}}, Addison-Wesley Publishing Co.,
  Reading, Mass.-London-Amsterdam, 1975. \MR{430163}

\bibitem{LeeBook2003SmoothManifold}
John~M. Lee, \emph{Introduction to smooth manifolds}, second ed., Graduate
  Texts in Mathematics, vol. 218, Springer, New York, 2013. \MR{2954043}

\bibitem{LiHong-Quan2006}
Hong-Quan Li, \emph{Estimation optimale du gradient du semi-groupe de la
  chaleur sur le groupe de {H}eisenberg}, J. Funct. Anal. \textbf{236} (2006),
  no.~2, 369--394. \MR{MR2240167 (2007d:58045)}

\bibitem{LiHong-QuanZhang2019}
Hong-Quan Li and Ye~Zhang, \emph{Revisiting the heat kernel on isotropic and
  nonisotropic {H}eisenberg groups*}, Comm. Partial Differential Equations
  \textbf{44} (2019), no.~6, 467--503. \MR{3946611}

\bibitem{Strichartz1986a}
Robert~S. Strichartz, \emph{Sub-{R}iemannian geometry}, J. Differential Geom.
  \textbf{24} (1986), no.~2, 221--263. \MR{862049 (88b:53055)}

\bibitem{VaropoulosBook1992}
N.~Th. Varopoulos, L.~Saloff-Coste, and T.~Coulhon, \emph{Analysis and geometry
  on groups}, Cambridge University Press, Cambridge, 1992. \MR{95f:43008}

\bibitem{Villani2009}
C\'{e}dric Villani, \emph{Hypocoercivity}, Mem. Amer. Math. Soc. \textbf{202}
  (2009), no.~950, iv+141. \MR{2562709}

\bibitem{VillaniOptimalTransportBook}
C{\'e}dric Villani, \emph{Optimal transport}, Grundlehren der Mathematischen
  Wissenschaften [Fundamental Principles of Mathematical Sciences], vol. 338,
  Springer-Verlag, Berlin, 2009, Old and new. \MR{2459454 (2010f:49001)}

\bibitem{Vinberg1988}
\`E.\~B. Vinberg, V.~V. Gorbatsevich, and O.~V. Shvartsman, \emph{Discrete
  subgroups of {L}ie groups}, Current problems in mathematics. {F}undamental
  directions, {V}ol.\ 21 ({R}ussian), Itogi Nauki i Tekhniki, Akad. Nauk SSSR,
  Vsesoyuz. Inst. Nauchn. i Tekhn. Inform., Moscow, 1988, pp.~5--120, 215.
  \MR{968445}

\bibitem{WangFY2016a}
Feng-Yu Wang, \emph{Derivative formulas and {P}oincar\'e inequality for
  {K}ohn-{L}aplacian type semigroups}, Sci. China Math. \textbf{59} (2016),
  no.~2, 261--280. \MR{3454046}

\bibitem{Zhang2023}
Ye~Zhang, \emph{On the {H}.-{Q}. {L}i inequality on step-two {C}arnot groups},
  C. R. Math. Acad. Sci. Paris \textbf{361} (2023), 1107--1114. \MR{4659071}

\end{thebibliography}
\end{document}